\documentclass[12pt]{article}
\usepackage{amsmath}
\usepackage{mathrsfs} 
\usepackage{fixmath}
\usepackage{amsthm}
\usepackage{mathtools}
\usepackage{amsfonts}
\usepackage{slashed}
\usepackage{amssymb}

\usepackage{latexsym}
\usepackage{tikz} 
\usepackage{esint} 
\usepackage{bm}
\usepackage{tikz}
\usepackage{tikz-cd}

\usepackage[matrix,tips,graph,curve]{xy}



\makeatletter 
\@addtoreset{equation}{section}
\makeatother

\theoremstyle{plain}
\newtheorem{theorem}[equation]{Theorem}
\newtheorem{corollary}[equation]{Corollary}
\newtheorem{lemma}[equation]{Lemma}
\newtheorem{proposition}[equation]{Proposition}
\newtheorem{conjecture}[equation]{Conjecture}

\theoremstyle{definition}
\newtheorem{definition}[equation]{Definition}

\newcommand{\IC}{\mathbb{C}}

\newcommand{\IH}{\mathbb{H}}

\newcommand{\IR}{\mathbb{R}}

\newcommand{\tr}{\mathrm{tr}}

\renewcommand{\deg}{\mathrm{deg}}

\newcommand\widebar{\overline}

\def\d/{/\mspace{-6.0mu}/}

\begin{document}
\title{On the extended Bogomolny equations on $\IR^2 \times \IR^+$ with real symmetry breaking}
\author{Weifeng Sun}
\date{}

\maketitle

\section*{Abstract}

In this paper, we construct solutions to the extended Bogomolny equations on $X = \IR^2 \times \IR^{+}$ with certain boundary conditions and asymptotic conditions.\\

Let $y$ be the coordinate of $\IR^{+}$. Roughly, both the boundary condition and the asymptotic condition say that a configuration (variables in the extended Bogomolny equations) approaches to certain model solutions when $y \rightarrow 0$ and $y \rightarrow +\infty$ resepctively. The boundary condition ($y \rightarrow 0$) is called ``generalized Nahm pole boundary condition" and the asymptotic condition ($y \rightarrow +\infty$) is called ``real symmetry breaking condition".\\


For each triple of polynomials with complex coefficients $(P(z), Q(z), R(z))$ with $\deg R < \deg Q$, and $Q, R$ are co-prime, $P, Q$ are monic, we construct a solution. This solution should be thought as an analog of the instanton solutions that Taubes created in \cite{Taubes2021TheR} or the solutions that Dimakis created in \cite{Dimakis2022TheField}, while their solutions satisfy a different asymptotic condition as $y \rightarrow +\infty$.\\

The idea of the construction is almost identical with Dimakis' construction in \cite{Dimakis2022TheField}. The main difference 
 is: We rely on the newly constructed model solution that satisfies the new asymptotic condition as the starting point.\\
 
 Section 1 is a breaf introduction on background settings and terminologies. Section 2 constructs the model solution. Section 3 uses an analog of Dimakis' argument to construct more solutions based on the model solution.

\paragraph{Acknowledgement}
The author thanks Panagiotis Dimakis for explaining his work in \cite{Dimakis2022TheField} and suggesting the author to work on this topic. The author is also very grateful to Siqi He for carefully and patiently explaining many technical details in his paper (with Rafe Mazzeo) \cite{He2020TheII}. The author also thanks Clifford Taubes and Rafe Mazzeo for many helpful discussions and useful advice.

\tableofcontents

\section{Introduction}\label{Section: Introduction}

In this section, we set up the basics for the extended Bogomolny equations. We need some preparation in linear algebra, which are summarized in appendix \ref{Appendix: The linear algebra}. Readers are supposed to be familiar with them.

\subsection{The extended Bogomolny equations}\label{Subsection: The extended Bogomolny equations}

Suppose $X = \IR^2 \times (0, +\infty)$. Suppose $x_1, x_2$ and $y$ are coordinates of $X$, where $y$ being the one for $\IR^{+}$. We use $\partial_1, \partial_2, \partial_y$ to represent the partial derivatives. Sometimes we treat $\IR^2$ as $\IC$ and let $z = x_1 + ix_2$ be its coordinate.\\

One way to introduce the extended Bogomolny equations on $X$ is as follows: It is a set of equations on configurations. Each configuration is a collection of 6 $su(2)$-valued functions $\Psi = (A_1, A_2, A_y, \Phi_1, \Phi_2, \Phi_3)$ on $X$. Here are the equations:

\begin{equation} 
\left\{
\begin{array}{lr}
\partial_1 A_2 - \partial_2 A_1 + [A_1, A_2] - [\Phi_1, \Phi_2] - \partial_y \Phi_3 - [A_y, \Phi_3] = 0 , &\\
 \partial_y \Phi_1 + [A_y, \Phi_1] + [\Phi_2, \Phi_3] = 0, &\\
\partial_y \Phi_2 + [A_y, \Phi_2] + [\Phi_3, \Phi_1] = 0, &\\
\partial_1 A_y - \partial_y A_1 + [A_1, A_y] + \partial_2 \Phi_3 + [A_2, \Phi_3]  = 0, &\\
\partial_2 A_y - \partial_y A_2 + [A_2, A_y] -  \partial_1 \Phi_3 - [A_1, \Phi_3] = 0, &\\
 -  \partial_2 \Phi_1 - [A_2, \Phi_1] + \partial_1 \Phi_2 + [A_1, \Phi_2] = 0, &\\
\partial_1\Phi_1 + [A_1, \Phi_1] + \partial_2 \Phi_2 + [A_2, \Phi_2] = 0. &
\end{array}
\right.
\end{equation} 

\textbf{Remark:} If the background metric on $X$ is not the Euclidean metric, then there should be additional terms in the extended Bogomolny equations that come from the metric. But this is not what we focus on in this paper.\\

These equations are $SU(2)$ gauge invariant. Moreover, all bullets of (1.1) except the first one are also $SL(2, \IC)$ gauge invariant. Here an $SU(2)$ (or $SL(2,\IC)$) gauge transformation is represented by an $SU(2)$-valued (or $SL(2,\IC)$-valued) function $g$ that sends

$$(A_1 + iA_2, A_y - i \Phi_3, \Phi_1 - i\Phi_2)$$ 
to
$$(g(A_1 + iA_2)g^{-1} - ((\partial_1 + i\partial_2)g )g^{-1}, g(A_y - i \Phi_3)g^{-1} - (\partial_y g) g^{-1}, g(\Phi_1 - i\Phi_2)g^{-1}).$$

An easy way to check that the gauge invariant properties of the equations is as follows: Define three operators acting on sections of a trivial $SL(2, \IC)$ bundle (that is to say, the bundle is just the trivial $\IC^2$ bundle but all the operators keep the $SL(2, \IC)$ structure):
$$D_1 = \partial_1 + i\partial_2 + (A_1 + iA_2), ~ D_2 = \Phi_1 - i\Phi_2, ~~ D_3 = \partial_y + A_y - i\Phi_3. $$

Then all but the first bullet of the extended Bogomolny equations can be written as
$$[D_1, D_2] = [D_2, D_3] = [D_3, D_1] = 0, $$

which is clearly $SL(2, \IC)$ gauge invariant (as commutators of operators).\\

The first bullet of the extended Bogomolny equations is equivalent to
$$[D_1, D_1^*] + [D_2, D_2^*] + [D_3, D_3^*] = 0, $$

where $*$ represents the ad-joint. Then it is obviously $SU(2)$ gauge invariant. It is not $SL(2, \IC)$ gauge invariant in general because  the $*$ operator in the equation relies on the $SU(2)$ structure.\\

It is proposed that the extended Bogomolny equations on $X$ are related with certain type of Kapustin-Witten equations, which are further conjectured to be related with the Jones polynomial of knots. The details of this proposal may be found in \cite{Gaiotto2011KnotTheory}\cite{Witten2011FivebranesKnots}\cite{Witten2012KhovanovTheory} \cite{Witten2014TwoHomology} \cite{Witten2016TwoHomology}  . Some relevant studies include \cite{Dimakis2022ModelEquations} \cite{Dimakis2022TheField}\cite{Gagliardo2014GeometricEquations}\cite{He2019ClassificationR+}\cite{He2019OpersEquations} \cite{He2019TheCondition}\cite{He2020TheII}  \cite{Mazzeo2013TheCondition} \cite{Mazzeo2017TheKnots}   \cite{Taubes2018SequencesEquations} \cite{Taubes2019TheAsymptotics} \cite{Taubes2020LecturesY} \cite{Taubes2021TheR} . The author does not attempt to list all the relevant references here.\\

According to the proposal, an interesting solution $\Psi$ should satisfy the so-called ``generalized Nahm pole boundary condition" as $y \rightarrow 0$ and one of several certain asymptotic conditions as $y \rightarrow +\infty$.\\

Full descriptions of the generalized Nahm pole condition can be found  in \cite{Gaiotto2011KnotTheory} and \cite{Mazzeo2013TheCondition} \cite{Mazzeo2017TheKnots}. Several different but essentially equivalent definitions occurred in these literature. Here is what we choose to use in this paper: When $y \rightarrow 0$, in a certain $SU(2)$ gauge, $|\Psi - \Psi_{\text{mod}}| = O(y^{-1 + \epsilon})$ for some $\epsilon > 0$, where $\Psi_{\text{mod}}$ is a model solution which already satisfies the generalized Nahm pole boundary condition (to be described later) that goes like $O(y^{-1})$ when $y \rightarrow 0$. Note that we do not require that the inequality is uniform in $z$.\\

We choose finitely many points with degrees (a positive integer assigned to each point) in $\IR^2$. These points are called ``knotted points". If the choice of points are given, then according to proposition 6.1 in \cite{He2019TheCondition}, there is a special solution to the extended Bogomolny equations that satisfies the generalized Nahm pole boundary condition and $|\Psi| \rightarrow 0$ as $y \rightarrow + \infty$. This special solution is the model solution that we use. Note that strictly speaking, different choices of knotted points give different versions of the ``generalized Nahm pole boundary condition".\\

On the other hand, there are three types of asymptotic conditions that we are interested in:

\begin{itemize}
    \item The first type says that, under a certain $SU(2)$ gauge, $|\Psi| \rightarrow 0$ as $y \rightarrow +\infty$. Solutions with this type of asymptotic condition are well studied in Section 6 of \cite{He2019TheCondition}, \cite{Taubes2021TheR}, \cite{Taubes2020LecturesY} and \cite{Dimakis2022TheField}.
    \item The second type of asymptotic condition says that $$|A_1|, |A_2|, |A_y|, |\Phi_1|, |\Phi_2| \rightarrow 0, ~~~|\Phi_3| \rightarrow 1$$ as $y \rightarrow +\infty$. This condition is called ``\textbf{real symmetry breaking condition}" in Gaiotto and Witten's paper \cite{Gaiotto2011KnotTheory}. And this is what we study in this paper. Sometimes we assume a stronger condition which says that:
    $$A_1, A_2, A_y, \Phi_1, \Phi_2 = O(y^{\epsilon}), \Phi_3 = \sigma + O(y^{\epsilon}), $$
    where $\epsilon$ is a positive real number, $\sigma$ is a constant $su(2)$ element with norm $1$.
    \item The third type of asymptotic condition is to require $|A_1|, |A_2|, |A_y|, |\Phi_3| \rightarrow 0$, but $\Phi_1$ and $\Phi_2$ approach the same constant $su(2)$ element with norm $1$. This is called ``\textbf{complex symmetry breaking condition}" in \cite{Gaiotto2011KnotTheory}. The author hopes to study it in the future.
\end{itemize}

\subsection{The metric representation} \label{Subsection: The metric representation}

 Let $E$ be the trivial $\IC^2$ bundle over $X$. It is convenient to use a pair $(H, \varphi)$ to represent (at least locally) a configuration, where $H$ is an $SU(2)$ Hermitian metric on $E$, $\varphi$ is an $sl(2, \IC)$ matrix whose items are holomorphic functions.\\

 Suppose $\varphi$ is written as
 $$ \varphi = \begin{pmatrix}
    A(z) & B(z) \\ P(z) & -A(z)
\end{pmatrix}, $$
where $A, B, P$ are holomorphic functions. Typically we assume they are polynomials. Then there is a configuration $\Psi_{\varphi}$ which trivially satisfies all but the first bullets of the extended Bogomolny equations, described by

 $$A_1 = A_2 = A_y = \Phi_3 = 0, ~ \Phi_1 - i\Phi_2 = \varphi = \begin{pmatrix}
    A(z) & B(z) \\ P(z) & -A(z)
\end{pmatrix}.$$ 

Any $SL(2, \IC)$ gauge transformation $g$ sends $\Psi_{\varphi}$ to another configuration $\Psi$ which also satisfies all but the first bullets of the extended Bogomolny equations. (Because they are $SL(2, \IC)$ gauge invariant.) Let $H = g^* g$ be the Hermitian metric. (See appendix \ref{Appendix: The linear algebra}.) Then we say $\Psi$ corresponds to the pair $(H, \varphi)$. We call the pair $(H, \varphi)$ a \textbf{metric representation} of $\Psi$. Note that this is \textbf{not} a $1$-$1$ correspondence. In fact,

\begin{itemize}
    \item Each configuration $\Psi$ can be locally represented by a pair $(H, \varphi)$ if and only if it satisfies all but the first bullets of the extended Bogomolny equations.
   \item It is possible that two different pairs $(H_1, \varphi_1)$ and $(H_2, \varphi_2)$ represent the same configuration $\Psi$. This happens if and only if there is a holomorphic $SL(2, \IC)$ gauge transformation (represented by a matrix $g$ whose items are holomorphic functions in $z$ and independent with $y$), such that
     $$H_1 = g^{*}H_2g, ~~ \varphi_1 = g^{-1}\varphi_2g.$$

     In fact, we may use different pairs of $(H, \varphi)$ at different regions to represent the same configuration. They are connected by $SL(2, \IC)$ holomorphic gauge transformations in the overlaps. So technically speaking, a configuration $\Psi$ corresponds to a Cech cocycle of pairs $(H, \varphi)$. 
     \item It is also possible that two different $\Psi_1$ and $\Psi_2$ correspond to the same pair $(H, \varphi)$. This happens if and only if $\Psi_1$ and $\Psi_2$ are $SU(2)$ gauge equivalent.  
\end{itemize}

Here is the illustration of the above statements:

 \paragraph{From $\Psi$ to $(H, \varphi)$}~\\

 Suppose $\Psi$ is a configuration such that $[D_1, D_2] = [D_2, D_3] = [D_3. D_1] = 0$. Keep in mind that the $SL(2, \IC)$ structure of $E$ is preserved under the operators $D_1, D_2, D_3$, whose definitions are in the last subsection.\\

Fix any $y = y_0 > 0$. Then $D_1$ can be viewed as a d-bar operator on sections of $E|_{y = y_0}$. So it gives $E|_{y = y_0}$ a holomorphic structure. Note that since $[D_1, D_2] = [D_1, D_3] = 0$, both $D_2$ and $D_3$ keeps this holomorphic structure. \\

On the $y = y_0$ slice, we may choose (at least locally) two holomorphic sections of $E$, namely $s_1$ and $s_2$. Since $D_1$ keeps the $SL(2, \IC)$ structure, we may assume that $s_1 \wedge s_2 = 1$ everywhere.\\

We may identify the choices of $s_1, s_2$ at different $y$ slices by requiring $D_3 s_1 = D_3 s_2 = 0$. Then $s_1$ and $s_2$ form a basis of $E$. Under this basis, $D_1$ is just $2\bar{\partial} = \partial_1 + i\partial_2$ and $D_3$ is just $\partial_y$. The operator $D_2$ is represented by an $sl(2, \IC)$ matrix whose items are holomorphic in $z$ and doesn't depend on $y$. We call this  matrix $\varphi$. Note that $[D_1, D_2] = [D_2, D_3] = [D_3, D_1] = 0$ is obvious in this basis.\\





Note that the definition of $(H, \varphi)$ depends on a choice of the holomorphic basis $s_1, s_2$. Different choices of the basis can be related by a holomorphic $SL(2, \IC)$ gauge transformation $g$ (an $SL(2, \IC)$ matrix whose items are holomorphic functions in $z$) which sends $(H, \varphi)$ to $(g^*Hg, g^{-1} \varphi g)$.\\


A warning: The holomorphic $SL(2, \IC)$ gauge transformation $g$ is \textbf{not} an actual gauge transformation on the original configuarition $\Psi$. In fact, it doesn't affect $\Psi$ at all. It only changes $(H, \varphi)$ by changing the pairs $(s_1, s_2)$ that we use to define it. Readers should not get confused by the (somehow misleading) term ``gauge transformation" that we use here.

\paragraph{From $(H,\varphi)$ to $\Psi$}~\\

Suppose we have a pair $(H, \varphi)$. Recall that the ``trivial" configuration  $\Psi_{\varphi}$ satisfies all but the first bullets of the extended Bogomolny equations.\\

We may choose an $SL(2, \IC)$ gange transformation $g$ such that $g^*g = H$. Then $g$ sends $ \Psi_{\varphi}$ to some configuration $\Psi$. Clearly $\Psi$ corresponds to $(H, \varphi)$.\\

Note that adding another $SU(2)$ gauge transformation $u$ to $g$ doesn't change $H = g^* g$. In fact, if we replace $g$ by $ug$, we still have
$$(ug)^* (ug) = g^*g = H.$$

So essentially, what we get is a configuration $\Psi$ up to an $SU(2)$ gauge transformation.\\

For each pair $(H, \varphi)$, there is a preferred way to choose $g$ which we'll frequently use in this paper: Each $SU(2)$ Hermitian metric $H$ can be written in the following format:

$$H = \begin{pmatrix}
    h + h^{-1}|w|^2 & h^{-1} \bar{w}\\
h^{-1} w & h^{-1}
\end{pmatrix}, $$
where $h$ is a positive real-valued function and $w$ is a complex-valued function.  Then we have $H = g^* g$, where $g = \begin{pmatrix}
 h^{\frac{1}{2}} & 0 \\
 h^{- \frac{1}{2}}w & h^{-\frac{1}{2}}
\end{pmatrix}$. Under the $SL(2, \IC)$ gauge transformation given by this specifically chosen $g$, the trivial configuration $\Psi_{\varphi}$ becomes

\begin{equation*} 
\left\{
\begin{array}{lr}
\Phi = \Phi_1 - i \Phi_2 = g \varphi g^{-1} = \begin{pmatrix}
    A - wB & hB \\ h^{-1}(wA + P - w^2 B + wA) & wB -A
\end{pmatrix},&\\
\Phi_3 = \dfrac{i}{2h} \begin{pmatrix}
-\partial_y h & -\partial_y \bar{w} \\
-\partial_y w &  \partial_y h
\end{pmatrix}, &\\
A_y = \dfrac{1}{2h}\begin{pmatrix}
0 & \partial_y \bar{w} \\
- \partial_y w & 0
\end{pmatrix} , &\\
A_1 = \dfrac{1}{2h}\begin{pmatrix}
- i \partial_2h &  2\partial \bar{w}\\
- 2\bar{\partial} w & i \partial_2 h
\end{pmatrix},&\\
A_2 = \dfrac{i}{2h} \begin{pmatrix}
 \partial_1 h &   2\partial \bar{w} \\
 2\bar{\partial}w & - \partial_1 h
\end{pmatrix} .
\end{array}
\right.
\end{equation*} 

 We use $\Psi_{H, \varphi}$ to denote this particular configuration that corresponds to $(H, \varphi)$. Any other configuration that corresponds to $(H, \varphi)$ is $SU(2)$ gauge equivalent to $\Psi_{H, \varphi}$.\\

For each $H = \begin{pmatrix}
    e^u + e^{-u}|w|^2 & e^{-u}\bar{w} \\ e^{-u}w & e^{-u}
\end{pmatrix}$ and $\varphi$, we write down the concrete formulas of $V(H, \varphi): = V(\Psi_{H, \varphi})$ in terms of $H$ and $\varphi$ in two special situations:

\paragraph{Special case 1}

$$H = \begin{pmatrix}
    e^u + e^{-u}|w|^2 & e^{-u}\bar{w} \\ e^{-u}w & e^{-u}
\end{pmatrix}, ~~ \varphi = \begin{pmatrix}
    0 & 0 \\ P(z) & 0
\end{pmatrix}.$$

Then $V(H, \varphi) = \dfrac{1}{2}\begin{pmatrix}
    E & \bar{F} \\ F & -E
\end{pmatrix}$, with

$$ E  = \Delta u + e^{-2u}(4|\bar{\partial} w|^2 + |\partial_y w|^2 + |P|^2),~~  F =
     e^{-u}(\Delta w - 2(\partial_y u)(\partial_y w) - 8 (\bar{\partial}w) (\partial u)). $$  

\begin{proof}
 
In this special case, suppose $h = e^u$. Then

\begin{equation*} 
\left\{
\begin{array}{lr}
\Phi = \Phi_1 - i \Phi_2 = \begin{pmatrix}
   0 & 0 \\ h^{-1} P & 0
\end{pmatrix},&\\
\Phi_3 = \dfrac{i}{2h} \begin{pmatrix}
-\partial_y h & -\partial_y \bar{w} \\
-\partial_y w &  \partial_y h
\end{pmatrix}, &\\
A_y = \dfrac{1}{2h}\begin{pmatrix}
0 & \partial_y \bar{w} \\
- \partial_y w & 0
\end{pmatrix} , &\\
A_1 = \dfrac{1}{2h}\begin{pmatrix}
- i \partial_2h &  2\partial \bar{w}\\
- 2\bar{\partial} w & i \partial_2 h
\end{pmatrix},&\\
A_2 = \dfrac{i}{2h} \begin{pmatrix}
 \partial_1 h &   2\partial \bar{w} \\
 2\bar{\partial}w & - \partial_1 h
\end{pmatrix} .
\end{array}
\right.
\end{equation*}

The first bullet of the extended Bogomolny equations is:
$$V(H, \varphi) = (\partial_1 A_2 - \partial_2 A_1 + [A_1, A_2]) - (\partial_y \Phi_3 + [A_y, \Phi_3] + [\Phi_1, \Phi_2]). $$

We have
\begin{equation*} 
\left\{
\begin{array}{lr}
    \partial_1 A_2 - \partial_2 A_1 + [A_1, A_2] &\\
        ~~~~ = \dfrac{(2i)}{h^2}\begin{pmatrix}
        h (\partial \bar{\partial} h) - |\bar{\partial} h|^2 + |\bar{\partial} w|^2 & h(\bar{\partial}\partial \bar{w}) - 2 (\partial \bar{w})(\bar{\partial}h)\\
        h(\partial \bar{\partial} w) - 2 (\bar{\partial}w) (\partial h) & - h (\partial \bar{\partial} h) + |\bar{\partial} h|^2 - |\bar{\partial} w|^2
        \end{pmatrix}, &\\
  \partial_y \Phi_3 + [A_y, \Phi_3] - [\Phi_1, \Phi_2] &\\
    ~~~~ = \dfrac{i}{2h^2} \begin{pmatrix}
    - h(\partial_y^2 h) + (\partial_y h)^2  - |\partial_y w|^2 - |P|^2 &
     - h (\partial_y^2 \bar{w}) + 2(\partial_y h)(\partial_y \bar{w})\\
    - h (\partial_y^2 w) + 2(\partial_y h)(\partial_y w)&
     h(\partial_y^2 h) - (\partial_y h)^2  + |\partial_y w|^2 + |P|^2
    \end{pmatrix}.
\end{array}
\right.
\end{equation*} 

So indeed $V(H, \varphi) = \dfrac{1}{2}\begin{pmatrix}
    E & \bar{F} \\ F & -E
\end{pmatrix}$, with

$$ E  =  h^{-2}(h \Delta h - |\nabla h|^2 + 4 |\bar{\partial} w|^2 + |\partial_y w|^2 + |P|^2) 
 = \Delta u + e^{-2u}(4|\bar{\partial} w|^2 + |\partial_y w|^2 + |P|^2),$$ $$ F = h^{-2}(h(\Delta w) - 2(\partial_y h)(\partial_y w) - 8 (\bar{\partial}w) (\partial h)) =
     e^{-u}(\Delta w - 2(\partial_y u)(\partial_y w) - 8 (\bar{\partial}w) (\partial u)). $$ 

\end{proof}

\paragraph{Special case 2}

$$H = \begin{pmatrix}
    e^u  & 0 \\ 0 & e^{-u}
\end{pmatrix}, ~~ \varphi = \begin{pmatrix}
    A(z) & B(z) \\ P(z) & -A(z)
\end{pmatrix}.$$

Then $V(H, \varphi) = \dfrac{1}{2} \begin{pmatrix}
    E & \bar{F} \\ F & -E
\end{pmatrix}$, with

$$ E  
 = \Delta u + e^{-2u} |P|^2 - e^{2u}|B|^2, ~~~  F   =  2e^u(B\bar{A}) - e^{-u}(A\bar{P} + P\bar{A}) . $$ 

\begin{proof}
    In this special case, suppose $h = e^u$. Then
\begin{equation*} 
\left\{
\begin{array}{lr}
\Phi = \Phi_1 - i \Phi_2  = \begin{pmatrix}
    A  & hB \\ h^{-1} P &  -A
\end{pmatrix},&\\
\Phi_3 = \dfrac{i}{2h} \begin{pmatrix}
-\partial_y h & 0 \\
0 &  \partial_y h
\end{pmatrix}, &\\
A_y = 0, &\\
A_1 = \dfrac{1}{2h}\begin{pmatrix}
- i \partial_2h &  0\\
0& i \partial_2 h
\end{pmatrix},&\\
A_2 = \dfrac{i}{2h} \begin{pmatrix}
 \partial_1 h &   0 \\
0 & - \partial_1 h
\end{pmatrix} .
\end{array}
\right.
\end{equation*} 

We have

\begin{equation*} 
\left\{
\begin{array}{lr}
    \partial_1 A_2 - \partial_2 A_1 + [A_1, A_2] &\\
        ~~~~ = \dfrac{(2i)}{h^2}\begin{pmatrix}
        h (\partial \bar{\partial} h) - |\bar{\partial} h|^2 & 0\\
        0 & - h (\partial \bar{\partial} h) + |\bar{\partial} h|^2 
        \end{pmatrix}, &\\
  \partial_y \Phi_3 + [A_y, \Phi_3] &\\
    ~~~~ = \dfrac{i}{2h^2} \begin{pmatrix}
    - h(\partial_y^2 h) + (\partial_y h)^2  &
     0\\
    0&
     h(\partial_y^2 h) - (\partial_y h)^2 
    \end{pmatrix}.
\end{array}
\right.
\end{equation*} 
And
$$[\Phi_1, \Phi_2] = \dfrac{1}{2i}[\Phi, \Phi^*] = \dfrac{1}{2i} (\begin{pmatrix}
    A  & hB \\ h^{-1} P &  -A
\end{pmatrix}\begin{pmatrix}
    \bar{A}  & h^{-1}\bar{P} \\ h\bar{B} &  -\bar{A}
\end{pmatrix}- \begin{pmatrix}
    \bar{A}  & h^{-1}\bar{P} \\ h\bar{B} &  -\bar{A}
\end{pmatrix}\begin{pmatrix}
    A  & hB \\ h^{-1} P &  -A
\end{pmatrix}) $$
$$ = \dfrac{1}{2i} \begin{pmatrix}
    h^2|B|^2 - h^{-2}|P|^2 & h^{-1}(A\bar{P} + \bar{P}A) - 2h(B\bar{A})\\
    h^{-1}(\bar{A}P + P\bar{A}) - 2h(\bar{B}A) & h^{-2}|P|^2 - h^2|B|^2
\end{pmatrix}. $$

So indeed
$$ E  =  h^{-2}(h \Delta h - |\nabla h|^2 - h^4|B|^2 + |P|^2) 
 = \Delta u + e^{-2u} |P|^2 - e^{2u}|B|^2,$$ $$ F = 2h(B\bar{A}) - h^{-1}(A\bar{P} + P\bar{A})  = 2e^u(B\bar{A}) - e^{-u}(A\bar{P} + P\bar{A}). $$

\end{proof}

\subsection{The deformation of a configuration}\label{Subsection: The deformation of a configuration}

In this paper, one general strategy to find a solution to the Bogomolny equations is: Construct a configuration $\Psi_0$ that satisfies all but the first bullet of the equations first. (This can usually be done using the metric representation described in subsection \ref{Subsection: The metric representation}.) And then seak an $SL(2, \IC)$ gauge transformation to make the first bullet vanish. So we need to study how the first bullet behaves under  $SL(2, \IC)$ gauge transformations.\\

Suppose $\Psi$ is a configration. We always assume it satisfies all but the first bullet of the extended Bogomolny equations. Recall that the first bullet can be written as

$$V(\Psi) := [D_1, D_1^*] + [D_2, D_2^*] + [D_3, D_3^*] = 0. $$

Suppose $u$ is an $SU(2)$ gauge transformation sending $\Psi$ to $u(\Psi)$. We may think $u$ as an $SU(2)$-valued function on $X$. Then $V(u(\Psi)) = uV(\Psi)u^{-1}$.  And the norm $|V(\Psi)|$ is invariant under this transformation.\\

On the other hand, suppose $g$ is an $SL(2, \IC)$ gauge transformation.  If in addition, $g^* = g$, then we call it a \textbf{Hermitian $SL(2, \IC)$ gauge transformation}.\\

The following fact is simple in linear algebra: For each $g \in SL(2, \IC)$, there is a unique $u\in SU(2)$ such that $ug$ is Hermitian. Because $ug(\Psi)$ and $g(\Psi)$ are $SU(2)$ equivalent. If we are not sensitive to $SU(2)$ transformations on what we get, we may only consider Hermitian $SL(2, \IC)$ gauge transformation when we deform $\Psi$.\\

Suppose $\Psi$ is a configuration. Given any section $s \in isu(2)$, there is a 1-parameter family of Hermitian $SL(2, \IC)$ gauge transformations $e^{ts}$ parameterized by $t \in \IR$. Since $sl(2, \IC) = su(2) \oplus isu(2)$, this deformation is perpendicular with $SU(2)$ gauge transformations initially.\\

On the other hand, any Hermitian $SL(2, \IC)$ valued function $g$ can be written as either $e^s$ or $- e^s$ for some $s \in isu(2)$ uniquely.  Note that $-e^s$ and $e^s$ act the same way on $\Psi$ as $SL(2, \IC)$ gauge transformations. So effectively any Hermitian $SL(2, \IC)$ gauge transformation can be represented by $e^s$ for some $s \in isu(2)$.

\paragraph{Deformation of the equation}~\\

Suppose $s$ is a section in $isu(2)$. We consider the Hermitian $SL(2, \IC)$
 gauge transformation $e^{s}$ acting on $\Psi$. What we get is written as $\Psi_s$. Then the first bullet of the extended Bogomolny equations is written as
 $$V(\Psi, s) : = V(\Psi_s) = \sum\limits_{i = 1}^3 [D_i -(D_i(e^s)) e^{-s}, D_i^* - ((D_i(e^s)) e^{-s})^*] $$
 $$ = \sum\limits_{i = 1}^3 [D_i - (D_i(e^s))e^{-s}, D_i^* + e^{-s}(D^*_i(e^s))],  $$
 where $D_1, D_2, D_3$ are given by $\Psi$. Note that the last inequality used the fact that $s^* = s$.\\

 This formula seems awful. But actually we only need two facts:

 \paragraph{The first fact} is
 $$V(\Psi, s) = V(\Psi) ~~ - ~~ (\gamma(s) + \gamma(-s)) (\Delta_{\Psi}s) ~~ + ~~ \dfrac{1}{2} (\gamma(s) - \gamma(-s)) ([V(\Psi), s]) ~~ + ~~ ``\text{remaining terms}", $$
 here the ``remaining terms" include multi-linear terms $R(s^{\otimes k} \otimes(D_i s)^{\otimes 2})$ and $R(s^{\otimes k} \otimes(D_i^* s)^{\otimes 2})$ with any integer $k \geq 0$ with convergent coefficients (comparible with the coefficients in the Taylor expansion of $e^s$). The operator $\Delta_{\Psi}$  behaves like the Laplacian, whose definition is
  $$\Delta_{\Psi} s : =  \sum\limits_{i = 1, 2, y} (\nabla_i^2s) + \sum\limits_{i = 1, 2, 3} ([\Phi_i, [\Phi_i, s]]),$$
  where $\nabla_i s = \partial_i s + [A_i, s]$, and $A_i, \Phi_i$ are six components of $\Psi$.

 \begin{proof} 
     We have $$V(\Psi_s) = \sum\limits_{i = 1}^3  [D_i - (D_i(e^s))e^{-s}, D_i^* + e^{-s}(D_i^*(e^s))] = \sum\limits_{i = 1}^3 ([D_i , D_i^*])$$ $$ + \sum\limits_{i = 1}^3(e^{-s}(D_iD_i^*(e^s)) + (D_i^*D_i(e^s))e^{-s}) + ``\text{remaning terms}"$$
     $$ = V(\Psi) + \sum\limits_{i = 1}^3 (\gamma(-s)(D_iD_i^*s) + \gamma(s)(D_i^*D_is)) + ``\text{remaining terms}". $$
Recall that $\nabla_1, \nabla_2, \nabla_y$ is the connection defined using $A_1, A_2, A_y$. And $\Phi = \Phi_1 - i\Phi_2$. We have $\nabla_i^* = \partial_i^* + A_i^* = - \nabla_i$. We have
 $$D_1D_1^*s + D_2D_2^* s + D_3D_3^*s$$ $$= (\nabla_1+i\nabla_2) (\nabla_1^* - i\nabla_2^*) s + [\Phi_1 - i\Phi_2, [\Phi_1^* + i\Phi_2^* ,s]] + (\nabla_y - i \Phi_3)(\nabla_y^* + i \Phi_3^*)s  $$
  $$= - \sum\limits_{i = 1, 2, y} (\nabla_i^2s) - \sum\limits_{i = 1, 2, 3} ([\Phi_i, [\Phi_i, s]]) - \dfrac{1}{2}[V(\Psi) ,s] = - \Delta_{\Psi}(s) - \dfrac{1}{2}[V(\Psi) ,s].$$

And

$$\sum\limits_{i = 1}^3 D_i^*D_is = \sum\limits_{i = 1}^3 (D_i^*D_is + [D_i, D_i^*]s) = - \Delta_{\Psi} s + \dfrac{1}{2}[V(\Psi), s]. $$

Thus the first fact follows directly.
  
 \end{proof}

 There are two alternative ways to write the first fact that will be useful later:

 $$V(\Psi, s) = V(\Psi) + \tilde{L}_{\Psi}(s) + \tilde{Q}(s) = V(\Psi) + L(s) + Q(s), $$
 where $\tilde{L}_{\Psi} = -(\gamma(s) + \gamma(-s))\Delta_{\Psi}$ is the second order term and 
 $\tilde{Q}(s)$ is the lower order terms; $L_{\Psi} = - \Delta_{\Psi}$ is the linear term and $Q(s)$ contains all the non-linear terms.

 \paragraph{The second fact} is a Weitzenbock type of formula:

$$<V(\Psi_s) - V(\Psi), s> = - \Delta(|s|^2) + \sum\limits_{i = 1}^3 2|v(-2s)D_i^*(s)|^2, $$
where the undefined linear algebra notations can be found in appendix \ref{Appendix: The linear algebra}.\\

This is a very mysteries (at least to the author) and happy formula. It comes from \cite{He2020TheII} as proposition 5.1 there (with some adaptions). It may be also more or less originated from Donaldson, Uhlenbeck and Yau's series of famous work on Kobayashi-Hitchin correspondence. Note that it is only true when $s \in isu(2)$, that is, Hermitian. It is generally false if only $s \in sl(2, \IC)$.

\begin{proof} $$<V(\Psi_s) - V(\Psi), s> =  <\sum\limits_{i = 1}^3 ([D_i - (D_i(e^s))e^{-s}, D_i^* + e^{-s}(D_i^* (e^{s}))] - [D_i, D_i^*]), s> $$ 
$$ = \sum\limits_{i = 1}^3 <(D_i(e^{-s}D_i^*(e^s)) + D_i^*(D_i(e^s)e^{-s}) - [(D_i(e^s)e^{-s}, e^{-s}D_i^*(e^s))]), s>.$$
The following identity can be checked directly:
$$D_i(e^{-s}D_i^*(e^s)) + D_i^*(D_i(e^s)e^{-s}) - [(D_i(e^s)e^{-s}, e^{-s}D_i^*(e^s))]$$ $$ = e^s(D_i(e^{-2s}D_i^*(e^{2s})))e^{-s} - (D_iD_i^*(e^s)- D_i^*D_i(e^s))e^{-s}. $$
Thus
$$<V(\Psi_s) - V(\Psi), s> = \sum\limits_{i = 1}^3 <e^s(D_i(e^{-2s}D_i^*(e^{2s})))e^{-s} - (D_iD_i^*(e^s)- D_i^*D_i(e^s))e^{-s}, s>$$
$$= \sum\limits_{i = 1}^3(<D_i(e^{-2s}D_i^*(e^{2s})), s>) - <[V(\Psi), e^{s}]e^{-s}, s>.$$
Here we used the fact that $Ad_{e^{s}}$ is self-adjoint and $Ad_{e^s}(s) = s$. Moreover,
$$<[V(\Psi), e^{s}]e^{-s}, s> = <\gamma(s)([V(\Psi), s]), s> = <[V(\Psi), s], s> = 0.$$
$$<D_1(e^{-2s}D_1^*(e^{2s})), s> = <(\partial_1 + i\partial_2)(e^{-2s}D_1^*(e^{2s})) + [A_1 + iA_2, (e^{-2s}D_1^*(e^{2s}))], s> $$
$$= \partial_1(<\gamma(-2s)(D_1^*(2s)) , s>) + \partial_2 (<i\gamma(-2s)(D_1^*(2s)), s>) + <\gamma(-2s)D_1^*(2s), D_1^*s> $$
$$=2\partial_1(<D_1^*(s) , s>) + 2\partial_2 (<iD_1^*s, s>) + 2|v(-2s)D_1^*(s)|^2. $$
The last step used the fact that $v(-2s)$ and $\gamma(-2s)$ are self-adjoint and $\gamma(-2s)(s) = s$. Similarly,
$$<D_2(e^{-2s}D_2^*(e^{2s})), s> = 2|v(-2s)D_2^*(s)|^2. $$
$$<D_3(e^{-2s}D_3^*(e^{2s})), s> = <\partial_y (e^{-2s}D_3^*(e^{2s})) + [A_y - i\Phi_3, (e^{-2s}D_3^*(e^{2s}))], s> $$
$$=2\partial_y(<D_3^*(s) , s>)  + 2|v(-2s)D_3^*(s)|^2. $$

Finally,
$$2\partial_1(<D_1^*(s) , s>) + 2\partial_2 (<iD_1^*s, s>) + 2 \partial_y(<D_3^*(s), s>) $$
$$= 2 \partial_1<- \partial_1s + i\partial_2s, s> + 2\partial_2<-i\partial_1s - \partial_2s, s> + 2\partial_y<-\partial_ys, s> = - \Delta(|s|^2). $$
So
$$<V(\Psi_s) - V(\Psi), s> = - \Delta(|s|^2) + \sum\limits_{i = 1}^3 2|v(-2s)D_i^*(s)|^2.$$
\end{proof}

\paragraph{Adding Hermitian $SL(2, \IC)$ gauge transformations}~\\

We need to be very careful to add two Hermitian $SL(2, \IC)$ gauge transformations. If we have two different Hermitian $SL(2, \IC)$ gauge transformations $e^{s_1}$ and $e^{s_2}$. If we apply $e^{s_2}$ first and then $e^{s_1}$ on a configuration $\Psi$, what we get is $e^{s_1}e^{s_2}(\Psi)$, which is generally not equivalent with $e^{s_1 + s_2} (\Psi)$.\\

Here is the correct way to add them: Let $\sigma$ satisfy $e^{2 \sigma} = e^{s_2} e^{2s_1}e^{s_2}$. Note that $e^{s_2} e^{2s_1}e^{s_2}$ is Hermitian. So $\sigma \in isu(2)$ exists and is unique.\\

Moreover, $(e^{s_1}e^{s_2}e^{-\sigma})^*(e^{s_1}e^{s_2}e^{-\sigma}) = I$. So $u = (e^{s_1}e^{s_2}e^{-\sigma})$ is in $SU(2)$. And $e^{s_1}e^{s_2} = u e^{\sigma}$. So $e^{s_1}e^{s_2}(\Psi)$ is $SU(2)$ equivalent with $e^{\sigma} (\Psi)$.\\

Sometimes, we still use the sloppy notation $s_1 + s_2$ to represent $\sigma$ mentioned above if there is not a potential confusion. Note that if $|\nabla^k s_1|$ and $|\nabla^k s_2|$ are all bounded for $k \leq K$, where $\nabla^k$ is the $k$th derivative and $K$ is a non-negative integer, then $|\nabla^K \sigma|$ is also bounded accordingly. This is the key that guarantees the elliptic estimates later are not ruined by the non-commutative way of taking sums.\\

One should also be careful with the differences here: For example, $s_1 + s_2 \neq s_2 + s_1$. And the bound of $|\nabla^K \sigma|$ may also depend on the lower order derivatives of $s_1$ and $s_2$.

\section{The model solution}\label{Section: The model solution}

This section studies a very special case in which the extended Bogomolny equations can be reduced to a single scalar equation. We assume the solution $\Psi$ corresponds to a pair $(H, \varphi)$ such that
$$H = \begin{pmatrix}
    e^u & 0 \\ 0 & e^{-u}
\end{pmatrix}, ~~~ \varphi = \begin{pmatrix}
    0 & 0 \\ \varphi(z) & 0
\end{pmatrix}. $$

In another word, $H$ can be diagonalized and $\varphi$ can be made as a lower triangular matrix at the same time. This section studies solutions to the extended Bogomolny equations in this special case with generalized Nahm pole boundary condition at $y \rightarrow 0$ and real symmetry breaking condition as $y \rightarrow +\infty$.\\

Based on the special case 1 in subsection \ref{Subsection: The metric representation},  the extended Bogomolny equations are reduced to a scalar equation:

$$\Delta u + e^{-2u} |P(z)|^2 = 0. ~~~~~ (*) $$

This section finds a solution to $(*)$ that gives us a solution to the extended Bogomolny equations with Nahm pole boundary condition as $y \rightarrow 0$ and real symmetry breaking condition as $y \rightarrow +\infty$. This solution will serve as a model solution to construct more solutions later in this paper.

\subsection{Boundary/assymptotic conditions}\label{Subsection: Boundary/assymptotic conditions}

In order for a solution of $(*)$ to represent a solution to the extended Bogomolny equations that we are interested in, we need to translate the  boundary/asymptotic conditions that we mentioned into conditions on $u$.

\paragraph{The generalized Nahm pole boundary condition as $y \rightarrow 0$}~\\

We define the generalized Nahm pole boundary condition in an indirect way: We quote Section 6 of \cite{He2019TheCondition} for a model solution that has the generalized Nahm pole boundary condition. The following theorem can be found as proposition 6.1 in \cite{He2019TheCondition}.

\begin{theorem}
For each non-zero polynomial $P(z)$, there is a unique solution to $(*)$, denoted as $u = u_0$, that has the following properties:
\begin{itemize}
    \item When $y \rightarrow 0$, it satisfies the generalized Nahm pole boundary condition.
    \item When $R \rightarrow +\infty$, $u_0 = N \ln R + \ln y + O(1)$, uniformly with respect to $\dfrac{y}{R}$. 
\end{itemize}

\end{theorem}

The following feature of $u_0$ will be useful later: On the entire $X$ (either $y \rightarrow 0$ or $y \rightarrow +\infty$), we have

$$e^{-u_0}|P(z)| = O(\dfrac{1}{y}). $$

Note that the definition of $u_0$ here differs from the one in some literature, say \cite{He2019TheCondition}, up to a sign. But it agrees with some other literature, say \cite{Dimakis2022TheField}.\\

In this paper, taking the above theorem for granted, we can simply define a generilized Nahm pole boundary conditions for a solution $u$ as follows:\\

Let $\Psi_0$ be the configuration that is represented by $u_0$. Then to say that a general configuration $\Psi$ satisfies the generalized Nahm pole boundary condition means: There is a positive number $\epsilon$ such that, possibly after an $SU(2)$ gauge transformation,

$$\Psi  = \Psi_0 + O(y^{-1 + \epsilon}), ~~ \text{as} ~~ y \rightarrow 0.$$

If we have a solution of $(*)$, then the following condition on $u$ implies that it represents a solution with the generalized Nahm pole boundary condition:

$$|\nabla (u - u_0)| = O(y^{-1 + \epsilon}), ~~~ u = u_0 + O(y^{\epsilon}), ~~~ \text{as} ~~ y \rightarrow 0. $$

\paragraph{Real symmetry breaking as $y \rightarrow +\infty$}

\begin{definition}
We say that a configuration $(A_1, A_2, A_y, \Phi_1, \Phi_2, \Phi_3)$ satisfies the real symmetry breaking condition if, possibly after an $SU(2)$ gauge transformation, for some $\epsilon > 0$,  when $y \rightarrow +\infty$
$$A_1, A_2, A_y, \Phi_1, \Phi_2 = O(y^{-\epsilon}), ~~ \Phi_3 = \dfrac{1}{2}\begin{pmatrix}
    -i & 0 \\ 0 & i 
\end{pmatrix} + O(y^{- \epsilon}). $$
\end{definition}

If we have a solution of $(*)$, then the following condition on $u$ implies that it represents a solution with the real symmetry breaking: As $y \rightarrow +\infty$,

$$\partial_1 u, \partial_2 u, e^{-u}|P(z)| = O(y^{-\epsilon}), ~~ \partial_y u = 1 + O(y^{-\epsilon}).$$

\subsection{The construction of a solution}\label{Subsection: The constructoin of a solution}

We use a version of the Perrons' method. This method is somehow standard in many PDE books, say \cite{Gilbarg2001EllipticOrder}.

\paragraph{Super/sub-solutions}~\\

A function $u$ such that

$$- \Delta u - e^{-2u}|P(z)|^2 \leq 0 $$

is called a sub-solution. A function $u$ such that

$$- \Delta u - e^{-2u}|P(z)|^2 \geq 0 $$

is called a super-solution.\\

We always assume $u$ is continuous, but allow the possibility that $u$ is not differentiable.  If this is the case, then a sub-solution/super-solution is defined in the weak sense: For example, to say

$$- \Delta u - e^{-2u}|P(z)|^2 \geq 0 $$

means for any non-negative smooth function $v$ on $X$ supported on a compact region, we have

$$\int_X (- u\Delta v - e^{-2u}|P(z)|^2v) \geq 0.  $$

If $u$ is both a sup-solution and a sub-solution, then it is a weak solution. The following argument is standard, showing that a weak solution is an actual smooth solution:\\

Suppose $u$ is a weak solution. Choose any ball whose closure is in $X$. Since $\Delta u + e^{-2u}|P(z)|^2 = 0$ in the ball and $e^{-2u}|P(z)|^2 \in L^2(B)$, using the standard elliptic regularity argument, we see that $u \in W^{2, 2}(B')$, where $B'$ is a slightly smaller ball in $B$, $W^{2, 2}$ is the Sobolev space. This further implies that $e^{-2u}|P(z)|^2 \in W^{2, 2}(B')$ using the Chain rule for derivatives, the fact that $u$ is continuous and bounded, and the Sobolev embedding/multiplication inequalities in a 3-dimensional space. This further implies that $u \in W^{2, 4}$ in a even smaller ball. We can do boot-strapping and conclude that $u \in W^{2, n}$ for any positive $n$ in a smaller ball. So it is smooth in that small ball. Since the ball is arbitrary chose, $u$ is smooth everywhere in $X$. In particular, it is an actual solution.\\

We construct a preferred super-solution $u_1$ and a preferred sub-solution $u_2$ such that: They both satisfy the generalized Nahm pole condition plus real symmetry breaking condition. And $u_1 \geq u_2$ pointwise.\\

Here is the definition of $u_1$:

$$u_1 = u_0 + y. $$

Here is the definition of $u_2$:

$$u_2 = u_0 + f(y),$$
where $f(y) = \begin{cases}
0,& y \leq C;\\
y - C\ln y - C - C\ln C, & y \geq C,
\end{cases}$\\

where $C$ is a large enough constant.\\

Clearly they all satisfy the desired boundary/assymptotic conditions. Moreover,

$$\Delta u_1 + e^{-2u_1} |P(z)|^2 = \Delta u_0 + e^{-2y}e^{-2u_0}|P(z)|^2 \leq \Delta u_0 + e^{-2u_0}|P(z)|^2 = 0.$$

So $u_1$ is indeed a super-solution.

$$\Delta u_2 + e^{-2u_2}|P(z)|^2 = \Delta u_0 + f''(y) + e^{-2f(y)} e^{-2u_0}|P(z)|^2 = (e^{-2f(y)} - 1)e^{-2u_0}|P(z)|^2 + f''(y). $$

Note that although $f''(y)$ doesn't exist, we do have

$$f''(y) \geq \begin{cases} 0, & y \leq C\\
\dfrac{C}{y^2}, & y > C
\end{cases}
$$ in the weak sense.\\

So replace $f''(y)$ by the above function on the right, we get $0$ when $y \leq C$ and we get

$$(e^{-2f(y)} - 1)e^{-2u_0}|P(z)|^2 + f''(y) \geq (e^{-2f(y)} - 1)e^{-2u_0}|P(z)|^2 + \dfrac{C}{y^2} \geq  - e^{-2u_0}|P(z)|^2 + \dfrac{C}{y^2} \geq 0, $$
when $y > C$.\\

Thus $u_2$ is indeed a sub-solution.\\

\paragraph{Perron's argument}~\\

In general, if we have a collection of functions $u \in \mathcal{U}$ that are bounded above point-wise, then we define $\sup \mathcal{U}$ to be another function whose value at each point $p$ is:
$$\sup \{u(p) ~ | ~ u \in \mathcal{U}\}. $$

For each point $p \in Z$, let $$u_3(p) = \sup\{u(p) ~ | ~ u ~ \text{is a sub-solution with} ~ u \leq u_2 ~ \text{on the entire} ~X\}. $$

We still use  $\mathcal{U}$ to denote the set $\{u(p) ~ | ~ u ~ \text{is a sub-solution with} ~ u \leq u_2 ~ \text{on the entire} ~X\}$.\\

Then clearly $u_3$ is a function that satisfies:

$$u_1 \leq u_3 \leq u_2. $$

We prove that $u_3$ is an actual solution. We need two lemmas first:

\begin{lemma}
If we have two sub-solutions $v_1, v_2$, then $\max\{v_1, v_2\}$ is still a sub-solution.
\end{lemma}

\begin{proof}
Clear $\max\{v_1, v_2\}$ is still continuous. Moreover, 
$$\Delta (\max\{v_1, v_2\}) \geq \Delta v_1, ~~~ \Delta (\max\{v_1, v_2\}) \geq \Delta v_2$$
in the weak sense. We write it as
$$\Delta (\max\{v_1, v_2\}) \geq \max\{\Delta v_1, \Delta v_2\}.$$
So
$$\Delta (\max\{v_1, v_2\}) + e^{-2\max\{v_1, v_2\}}|P(z)|^2 \geq \max\{\Delta v_1, \Delta v_2\} + e^{-2\max\{v_1, v_2\}}|P(z)|^2$$ $$\geq \min\{\Delta v_1 + e^{-2v_1}|P(z)|^2, \Delta v_2 + e^{-2v_2}|P(z)|^2\} \geq 0. $$

So $\{v_1, v_2\}$ is also a sub-solution.

\end{proof}

\begin{lemma}
    Suppose $B$ is a ball whose closure is a compact subset of $X$. Suppose $v\in \mathcal{U}$. Then there is a unique solution $v_0$ in $B$ that is continuous up to the boundary, such that $v_0 \geq v$ in $B$ and $v_0 = v$ on $\partial B$. We call it a ``lifting to an solution" of $v$ in $B$. Moreover $v \leq v_2$.
\end{lemma}

\begin{proof}
We show the uniqueness first. Suppose two actual solutions $v_0, v_0'$ in $B$ both equal $v$ on $\partial B$. Then 

$$\Delta(v_0 - v_0') + (e^{-2v_0} - e^{-2v_0'})|P(z)|^2 = 0. $$

Since $v_0 - v_0'$ is continuous up to boundary in $B$. If they are not the same in $B$, without loss of generality, we may assume that $p$ is a positive maximal point of $v_0 - v_0'$ inside of $B$ with the property

$$\Delta (v_0 - v_0') (p) <0.$$

Since we have $$ (e^{-2v_0} - e^{-2v_0'}) |P(z)|^2 \leq 0$$
at $p$, it contradicts with the fact that

$$\Delta(v_0 - v_0') + (e^{-2v_0} - e^{-2v_0'})|P(z)|^2 = 0. $$

To prove the existence, choose a large enough positive constant $C$. Using a standard Dirichlet argument, we may construct a function $v_1$ such that

$$\Delta v_1 - Cv_1 = - e^{-2v}|P(z)|^2 - Cv ~~ \text{in $B$ and} ~~ v_1 = v ~~ \text{on $\partial B$}.  $$
 Using the fact that $\Delta v_1 - Cv_1 \leq \Delta v - Cv$ and maximal principle, we know $v_1 \geq v$. Then successively construct a sequence $\{v_n\}$ that have the same boundary values on $\partial B$ and
$$\Delta v_{n+1} -C v_{n+1} = - e^{-2v_n}|P(z)|^2 - Cv_n$$
in $B$. Since $B$ and $P(z)$ are fixed and since $v$ is bounded below by $u_1$ from the construction, we may assume $C$ is large enough such that, if $v_n$ is also bounded from the below by $u_1$ and if $v_{n+1} \geq v_{n} $, then $$e^{-2v_{n+1}}|P(z)|^2 + Cv_{n+1} \geq e^{-2v_{n}}|P(z)|^2 + Cv_{n}.  $$
Thus 
$$v_{n+1} \geq v_n $$
implies
$$\Delta v_{n+2} - C v_{n+2} \leq \Delta v_{n+1} - Cv_{n+1},$$
which further implies
$$v_{n+2} \geq v_{n+1}. $$

So inductively we get $v_n$ is an increasing sequence and $\Delta v_n$ is a decreasing sequence.\\

On the other hand, we have $v \leq u_2$. So

$$e^{-2u_2}|P(z)|^2 + Cu_2 \geq e^{-2v}|P(z)|^2 + Cv. $$
$$\Delta u_2 - Cu_2 \leq - e^{-2u_2} - Cu_2 \leq -e^{-2v}|P(z)|^2 - Cv = \Delta v_1 - Cv_1. $$

And we know that $v_1 \leq u_2$ on the boundary. A maximal principle implies that $v_1 \leq u_2$ in $B$ as well. And inductively we get $v_n \leq u_2$ for all $n$. In particular, the increasing sequence $v_n$ has an upper bound $u_2$. So it converges uniformly to a continuous function $v_0$.\\

Clearly $v \leq v_0 \leq u_2$.\\

Finally, suppose $G$ is the Green's function of $\Delta$ for $B$ centered at any point $p\in B$ and $P$ is the Poisson's kernel on $\partial B$ evaluated at $p$. Since all $v_n$ are bounded uniformly in $\bar{B}$,  using dominated convergent theorem,
$$v_0(p) = \lim\limits_{n \rightarrow +\infty} v_{n+1}(p) = (\lim\limits_{n \rightarrow +\infty} \int_B G (Cv_{n+1} - e^{-2v_n}|P(z)|^2 - Cv_n)) + \int_{\partial B} Pv)$$ $$ = -\int_B G e^{-2v_0}|P(z)|^2 + \int_{\partial B} Pv_0.  $$

This implies that
$$\Delta v_0 + e^{-2v_0}|P(z)|^2 = 0 $$
in $B$. So $v_0$ is what we want.

\end{proof}

Note that in the above proof, suppose $v \in \mathcal{U}$. Then its lifting to a solution in $B$ (while keeping the outside part of $B$ unchanged) is still in $\mathcal{U}$.\\

Given the above two lemmas, we prove that $u_3$ is an actual sotluion:

\begin{proof}
    Consider a ball $B$ whose closure is a compact subset of $X$. Since elements in $\mathcal{U}$ are continuous and bounded above uniformly by $u_2$ in the closure of $B$, we may find an increasing sequence $v_n$ in $\mathcal{U}$ that converges to $u_3$ on the closure of $B$. Note that this implies that $u_3$ is also continuous in $B$.\\

    For each element in $\mathcal{U}$, replacing it with its ``lifting to a solution" in $B$ makes it larger without violating being in $\mathcal{U}$. Without affecting the argument, we may assume each $v_n$ that we chose equal to its ``lifting to a solution" in $B$. In particular, each $v_n$ is an actual solution in $B$.\\

    Finally, suppose $G$ is the Green's function of $\Delta$ for $B$ centered at any point $p\in B$ and $P$ is the Poisson's kernel on $\partial B$ evaluated at $p$. Since all $v_n$ are bounded uniformly in $\bar{B}$ and convergent uniformly to $u_3$ in $\bar{B}$,  using dominated convergent theorem,
$$u_3(p) = \lim\limits_{n \rightarrow +\infty} v_n(p) = \lim\limits_{n \rightarrow +\infty} (\int_B G (- e^{-2v_n}|P(z)|^2) + \int_{\partial B} Pv_n)$$ $$ = -\int_B G e^{-2u_3}|P(z)|^2 + \int_{\partial B} Pu_3.  $$

This implies that
$$\Delta v_3 + e^{-2v_3}|P(z)|^2 = 0 $$
in $B$. Since $B$ is arbitrary, $u_3$ is a solution on the entire $X$.
\end{proof}

\paragraph{The real symmetry breaking condition}~\\

Let $u_1, u_2, u_3$ have the same meaning as before. Remember that $u_1 \leq u_3 \leq u_2$. We verify that $u_3$ satisfies the real symmetry breaking condition as $y \rightarrow +\infty$. We always assume $\epsilon$ is a small enough constant, and $y \rightarrow +\infty$ (which means, all the inequalities only work when $y$ is large enough).\\

It suffices to verify that:

$$\partial_1 u_3, \partial_2 u_3, e^{-u_3}|P(z)| = O(y^{-\epsilon}),  ~~\partial_y u_3 = 1 + O(y^{-\epsilon}).$$

Since $u_3 \geq u_1 = u_0 + y$, 

$$e^{-u_3}|P(z)| \leq e^{-y}e^{-u_0}|P(z)|  = O(e^{-y}).$$

Let $v = u_3 - y - N\ln R$. It only remains to verify that

$$|\nabla v| = O(y^{-\epsilon}).$$

The equation for $v$ is:

$$ \Delta v + NR^{-2} + e^{-2u_3}|P(z)|^2 = 0.$$

Then

$$e^{-2u_3}|P(z)|^2 = O(e^{-2y}). $$

So 
$$\Delta v = O(y^{-2}). $$

Moreover, we have

$$|v| \leq \max\{|u_1 - y - N\ln R, u_2 - y - N\ln R|\} = O(\ln y). $$

We need a lemma:

\begin{lemma}\label{lemma 1}
    Choose a point $(z, y)$ first. Suppose $r \leq \dfrac{y}{4}$. Let $B_r$ be the ball of radius $r$ centered at $(z, y)$. Then
$$\int_{B_r} |\nabla v|^2 = O(r |\ln y|^2).$$
\end{lemma}

\begin{proof}
    For each ball $B_r$, we may choose a cut-off function $\chi$ that is supported in $B_{2r}$ (the ball of radius $2r$ with the same center) such that:
    $$\chi = 1 ~~ \text{in} ~~ B_r, ~~~ \text{and} ~~ |\nabla^k \chi| = O(r^{-k}), ~~ \text{for any non-negative integer $k$}.$$

Then 

$$\int_{B_{2r}}|\nabla(\chi v)|^2 = -\int_{B_{2r}}\Delta(\chi v) \cdot (\chi v) \leq \int_{B_{2r}} (|\Delta v||v| + 2|\nabla(\chi v)||\nabla \chi||v| + (|\Delta \chi| + 2|\nabla \chi|^2)|v|^2) $$
$$\leq \dfrac{1}{2} \int_{B_{2r}}|\nabla (\chi v)|^2 + \int_{B_{2r}} (|\Delta v||v| + \dfrac{C}{r^2}|v|^2), $$
where $C$ is a large enough constant.\\

Recall that

$$\Delta v = e^{-2v}|P(z)|^2 = O(y^{-2}), ~~ v = O(\ln y). $$
It implies
$$\int_{B_r} |\nabla v|^2 \leq \int_{B_{2r}}|\nabla(\chi v)|^2 \leq 2\int_{B_{2r}} (|\Delta v||v| + \dfrac{C}{r^2}|v|^2) = O(r^3y^{-2}|\ln y|) + O(r (\ln y)^2) = O(r (\ln y)^2).  $$

\end{proof}

Now we prove that $ |\nabla 
 v| = O(y^{-\epsilon})$.

 \begin{proof}

We assume $r \leq \dfrac{y}{8}$ and let $G$ be the Green's function of the Laplacian centered at $(z, y)$. Let $B_r$ be the ball of radius $r$ with the same center. Let $\chi$ be the same cut-off function as in the previous lemma. Then

$$(\nabla v)(z, y) = \int_{B_{2r}} G  \Delta (\chi \nabla v), $$
where $(\nabla v)(z, y)$ is $\nabla v$ evaluated at the point $(z, y)$. Using integration by parts, one can verify that
$$|(\nabla v)(z, y)| \leq  \int_{B_{2r}} G |\Delta (\chi \nabla v)| \leq \int_{B_{2r}}(8(|\nabla G| |(\nabla \chi)| + |G| |\nabla^2 \chi|) |\nabla v| + |\nabla (G\chi)| |\Delta v|). $$
We have $|\nabla G| |(\nabla \chi)| + |G| |\nabla^2 \chi| = O(r^{-3})$, so

$$\int_{B_{2r}}(8(|\nabla G| |(\nabla \chi)| + |G| |\nabla^2 \chi|) |\nabla v|)\leq \int_{B_{2r}}(8(|\nabla G| |(\nabla \chi)| + |G| |\nabla^2 \chi|)^2)^{\frac{1}{2}} (\int_{B_{2r}}|\nabla v|^2)^{\frac{1}{2}} $$
$$= O(r^{-\frac{3}{2}} \cdot r^{\frac{1}{2}} |\ln y|) = O(r^{-1} \ln y).  $$

Moreover, 

$$\int_{B_{2r}} |\nabla (G \chi)| = O(r^2). $$
So
$$\int_{B_{2r}}|\nabla(G\chi)| |\Delta v| = O(r^2 y^{-2}). $$

Remember that we are working on the region such that $y \rightarrow +\infty$. We may choose $r = \sqrt{y}$, then we get

$$|\nabla v| = O(r^{-1} |\ln y|) + O(r^2 y^{-2}) = O(y^{-\frac{1}{2}}|\ln y|) = O(y^{-\epsilon}),$$
where $0 < \epsilon < \dfrac{1}{2}$.
    
\end{proof}

\paragraph{The generalized Nahm pole boundary condition}~\\

We verify the generalized Nahm pole boundary condition for $u_3$. In this section, we always assume $\epsilon$ is a small enough constant and $y \rightarrow 0$ (which means, all inequalities only work when $y$ is small). Note that both $u_1$ and $u_2$ satisfy the generalized Nahm pole boundary condition.\\

Since $u_1 \leq u_3 \leq u_2$, cleary we have

$$u_3 = u_0 + O(y^{\epsilon}). $$

So it suffices to check that

$$|\nabla (u_3 - u_0)| = O(y^{-1 + \epsilon}). $$

Let $v = u_3 - u_0 = O(y)$. Then

$$\Delta v  + (e^{-2v} - 1) e^{-2u_0} |P(z)|^2 = 0. $$

So

$$\Delta v = - (e^{-2v} - 1) e^{-2u_0} |P(z)|^2 = O(y^{-1}). $$

The following lemma is similar with lemma \ref{lemma 1}, except that we are working in the region $y \rightarrow 0$ now and the definition of $v$ is also different.

\begin{lemma}
    Choose a point $(z, y)$ first. Suppose $r \leq \dfrac{y}{4}$. Let $B_r$ be the ball of radius $r$ centered at $(z, y)$. Then
$$\int_{B_r} |\nabla v|^2 = O(r y^2).$$
\end{lemma}

\begin{proof}
  We choose the same cut-off function supported in $B_{2r}$ and equal $1$ in $B_r$ as always. Then the same argument as in the last paragraph,

$$\int_{B_r} |\nabla v|^2 = O( \int_{B_{2r}} (|\Delta v||v| + \dfrac{1}{r^2}|v|^2)) = O(r^3) + O(r y^2) = O(ry^2).  $$
\end{proof}

Now we prove that $ |\nabla 
 v| = O(y^{-1 + \epsilon})$. This is also similar with the last subsection (the real symmetry breaking).

 \begin{proof}

We assume $r \leq \dfrac{y}{8}$ and let $G$ be the Green's function of the Laplacian centered at $(z, y)$. Let $B_r$ be the ball of radius $r$ with the same center. Let $\chi$ be the same cut-off function as always. Then

$$(\nabla v)(z, y) = \int_{B_{2r}} G  \Delta (\chi \nabla v), $$
where $(\nabla v)(z, y)$ is $\nabla v$ evaluated at the point $(z, y)$.\\

We have $$|\nabla G| |(\nabla \chi)| + |G| |\nabla^2 \chi| = O(r^{-3}), ~~~ \int_{B_{2r}} |\nabla (G \chi)| = O(r^2).$$ Using the same method as the real symmetry breaking case,

$$|(\nabla v)(z, y)| \leq \int_{B_{2r}}(8(|\nabla G| |(\nabla \chi)| + |G| |\nabla^2 \chi|)^2)^{\frac{1}{2}} (\int_{B_{2r}}|\nabla v|^2)^{\frac{1}{2}} + \int_{B_{2r}}|\nabla(G\chi)| |\Delta v|$$
$$= O(r^{-\frac{3}{2}} \cdot r^{\frac{1}{2}} y) + O(r^2 y^{-1}).  $$

 We may choose $r = \dfrac{y}{16}$. Then we get

$$|\nabla v| = O(r^{-\frac{3}{2}} \cdot r^{\frac{1}{2}} y) + O(r^2 y^{-1}) = O(1).  $$

    \end{proof}

\subsection{The uniqueness}\label{Subsection: The uniqueness}

Suppose $u_0, u_1, u_2, u_3$ all have the same meanings as in the last subsection. We prove that the solution of $(*)$ is unique under certain constraints:

\begin{proposition} \label{uniqueness proposition}
    Suppose $u$ is a solution of $(*)$ such that,

$$|u - u_1| \leq C(y^{\epsilon} + y^{1 - \epsilon}), $$
where $C$ is any fixed large constant and $\epsilon \in (0, 1)$ is another fixed real number. Then $u = u_3$.
\end{proposition}

The author conjectures under weaker conditions it is still unique.
\begin{conjecture}
    Suppose $u$ is a solution of $(*)$ such that, for some $\epsilon > 0$,
    \begin{itemize}
        \item When $y \rightarrow +\infty$, $e^{-u}|P(z)| = O(y^{-\epsilon})$ and $|\nabla(u - y)| = O(y^{-\epsilon})$.
        \item When $y \rightarrow 0$, $u = u_0 + O(y^{\epsilon})$ and $|\nabla(u - u_0)| = O(y^{-1 + \epsilon})$.
    \end{itemize}
    Then $u = u_3$. (The inequalities do not necessarily uniform in $z$ in the conjecture.)
\end{conjecture}

The remaining of this subsection proves proposition \ref{uniqueness proposition}

\begin{lemma} \label{comparison lemma}
    Suppose $v$ is a smooth function on $X$ such that for some $\epsilon > 0$ and $C > 0$,  $|v| \leq C(y^{\epsilon} + y^{1- \epsilon})$. Moreover, $\Delta v \geq 0$ on the entire $X$. Then $v \leq 0$. 
\end{lemma}

\begin{proof}
    Let $f(y) = \sup\limits_{z \in \IC}v(z, y)$. Note that since $v$ is bounded on each fixed $y$ slice, $f(y)$ is a well-defined function on $y \in [0, +\infty)$. Moreover, $f(y) \leq C(y^{\epsilon} + y^{1-\epsilon})$. In particular, we have $f(0) = 0$.\\

    Note that since $v$ is continuous, $f(y)$ is also a continuous function. We show that $f(y)$ is a concave up function.\\

    In fact, for any $\epsilon > 0$ and $y = y_0 > 0$. Let $z_0$ be a point such that
    $f(y_0) - v(y_0, z_0) < \epsilon$. Moreover, we may assume $v(z, y_0) - \epsilon |z - z_0|^2$ has a strict maximum over $z \in \IC$ at $z = z_0$. In particular, the $\partial_1^2 + \partial_2^2$ acting on $v(z, y_0) - \epsilon |z - z_0|^2$ has a negative value at $z = z_0$.\\

    Note that $\Delta(v(z, y_0) - \epsilon |z - z_0|^2) \geq - 2 \epsilon$. Since 
    $\Delta = \partial_1^2 + \partial_2^2 + \partial_y^2$. 
    This indicates that

    $$(\partial_y^2v)(z_0, y_0) \geq - 2 \epsilon. $$

    Since $\partial_y^2 v$ is smooth, we actually have $$(\partial_y^2v) > - 3 \epsilon $$
    in a neighbourhood of $(z_0, y_0)$. So using the mean value theorem,
    $$\liminf\limits_{h \rightarrow 0} \dfrac{v(z_0, y_0 - h) + v(z_0, y_0 + h) - 2v(z_0, y_0)}{h^2} = \partial_y^2 v(z_0, \xi) > -3\epsilon, $$
where $\xi$ is a number that can be arbitrarily close to $y_0$.\\

So

$$\liminf\limits_{h \rightarrow 0} \dfrac{f(y_0 - h) + f(y_0 + h) - 2f(y_0) + 2\epsilon}{h^2} > - 3 \epsilon.$$

Letting $\epsilon \rightarrow 0$, we get

$$\liminf\limits_{h \rightarrow 0} \dfrac{f(y_0 - h) + f(y_0 + h) - 2f(y_0)}{h^2} = 0. $$

Since this is true at any point, we know that $f(y)$ is a concave up function.\\

Fix any $y_0 > 0$. Then since $f(y)$ is concave up, for any other $y > 0$,

$$f(y)y_0 + f(0)(y - y_0) \geq f(y_0)y. $$

Recall that $f(0) = 0$, we have 

$$f(y) \geq \dfrac{f(y_0)}{y_0} y. $$

But $f(y) \leq C(y^{\epsilon} + y^{1 - \epsilon})$. We must have $f(y_0) \leq 0$. This is true for any $y_0$. Thus $f(y) \leq 0$ for any $y > 0$. And it implies that $v \leq 0$.

\end{proof}

\begin{corollary}
    If $v_1$ and $v_2$ are two solutions of $(*)$ on $X$ such that $0 \leq v_2 - v_1 \leq C(y^{\epsilon} + y^{1 - \epsilon})$ on the entire $X$ for some $C > 0$ and $\epsilon \in (0, 1)$. Then $v_1 = v_2$.
\end{corollary}

\begin{proof}
    We have
    $$\Delta (v_2 - v_1) = (e^{-2v_1} - e^{-2v_2})|P(z)|^2 \geq 0.  $$
So from the lemma \ref{comparison lemma}, $v_2 - v_1 \leq 0$. Hence $v_1 = v_2$.
\end{proof}

\begin{lemma} \label{a small lemma}
    Suppose $C$ is a large constant and $\epsilon \in (0, 1)$. Then there is a second order differentiable function $f(y)$ on $y \in [0, +\infty)$ with the following properties:
\begin{itemize}
    \item $f(0) = 0$.
    \item $f(y) \leq y - C(y^{\epsilon} + y^{1-\epsilon}). $
    \item $f''(y) + \dfrac{C}{y^2}(e^{-2f(y)} - 1) \geq 0.$
\end{itemize}
\end{lemma}

\begin{proof}
    In fact, we may choose a even larger $C'$ and let $$f(y) = y - C'(y^{\epsilon} + y^{1 - \epsilon}).$$
Then all three bullets are satisfied provided that $C'$ is large enough. Here is the reason: The first two bullets are obvious. For the third bullet, 

$$f''(y) + \dfrac{C}{y^2}(e^{-f(y)} - 1) = \dfrac{1}{y^2}( C'\epsilon (1 - \epsilon)(y^{\epsilon} + y^{1 - \epsilon}) + Ce^{-2f(y)} - C).  $$

when $y$ is small, $y \leq C'(y^{\epsilon} + y^{1 - \epsilon})$. So $e^{-2f(y)} - C \geq 0$ and hence the above expression is non-negative.\\

Otherwise, since $C'$ is large, when $y > C'(y^{\epsilon} + y^{1 - \epsilon})$, since $y$ has a lower bound,  we may assume 
$$C' \epsilon (1 - \epsilon) (y^{\epsilon} + y^{1 - \epsilon}) - C \geq 0. $$

So in any case, the third bullet is also true.

\end{proof}

Now here is the proof of proposition \ref{uniqueness proposition}.

\begin{proof}

 Suppose $u$ is a solution of $(*)$ such that,
$$|u - u_1| \leq C(y^{\epsilon} + y^{1 - \epsilon}).$$

In the Perron's argument, we may instead choose $u_1$ to be $u_0 + y + C(y^{\epsilon} + y^{1 - \epsilon})$ for a possibly larger constant $C$ and choose $u_2$ to be $u_0 + y - C(y^{\epsilon} + y^{1 - \epsilon})$. Since $C$ is assumed to be large, by lemma \ref{a small lemma}, one can verify that they are indeed super/sub solutions.\\

    Using these substituted $u_1$ and $u_2$, we may run Perron's argument again and get a solution, still call it $u_3$. But recall that
    $$u_3= \sup\{u ~ | ~ u ~ \text{is a sub-solution with} ~ u \leq u_2 ~ \text{on the entire} ~X\}. $$

And $u$ is a solution with $u \leq u_2$. So we have $u_3 \geq u$.\\

    On the other hand, we may modify the Perron's argument to define
    $$u_3' = \inf\{u(p) ~ | ~ u ~ \text{is a super-solution with} ~ u \geq u_1 ~ \text{on the entire} ~X\}. $$
    
By the same reason, $u_3'$ is also an actual solution. And $u_3' \leq u$. So in order to prove $u$ is unique, we only need to show that 

$$u_3 = u_3'.$$

Since we know that $u_3' \leq u_3$. And

$$\Delta (u_3 - u_3') = (e^{-2u_3} - e^{-2u_3'}) |P(z)|^2 \leq 0. $$

Moreover,
$$ u_3 - u_3' \leq |u_1 - u_2| \leq 2C(y^{\epsilon} + y^{1 - \epsilon}). $$

Thus by lemma \ref{comparison lemma}, we know that $u_3 = u_3'$. 
\end{proof}

\paragraph{A remark on poly-homogeneous expansion}~\\

Based on Section 6 of \cite{He2019TheCondition}, $u_0$ has a poly-homogeneous expansion on $\hat{X}$, where $\hat{X}$ is a preferred way to compactify $X$ as a manifold with boundaries and corners whose definition can be found in either appendix \ref{Appendix: Compactification of $X$} or \cite{He2019TheCondition}. However, surprisingly, this is not the case for $u_3$. In fact, $u_3$ may have a poly-homogeneous expansion on each boundary away from the corners of $\hat{X}$, but they do not seem to compatible at the corner. Even in the simplest case: Suppose there is no knot singularity as $y \rightarrow 0$, but only Nahm pole sigular boundary condition with real symmetry condition. The solution is written explicitly as

$$u = \ln(\dfrac{e^y - e^{-y}}{2}). $$

This solution doesn't seem to have a poly-homogeneous expansion at the corner of $\hat{X}$ given by $R \rightarrow +\infty$ while $\psi \rightarrow 0$ at the same time, where $y = R \sin \psi$.

\section{The continuity method}\label{Section: The continuity method}

In this section, we use the continuity method to construct more solutions to the extended Bogomolny equations with generalized Nahm pole boundary condition and the real symmetry condition. This construction is almost identical with Dimakis' argument in \cite{Dimakis2022TheField}. By way of looking ahead, here is a brief sketch:\\

With the help of the solution $u_3$ constructed in the last section, for each triple $(P(z), Q(z), R(z))$, where $P, Q, R$ are polynomials with $\deg R < \deg Q$, $Q$ and $R$ are coprime, $P, Q$ are monic, we construct an approximate solution $(H_*, \varphi)$ that corresponds to it.  Then we improve the approximate solution near $y \rightarrow 0$. Finally, we use a continuity method to further improve it to get an actual solution.

\subsection{The approximate solution}\label{Subsection: The approximate solution}

Suppose a triple $(P, Q, R)$ is given. The approximate solution constructed in this section will be a pair $(H_*, \varphi)$ with

$$\varphi = \begin{pmatrix}
    0 & 0 \\ P(z) & 0
\end{pmatrix}.$$

We define $H_*$ separately in different regions. Recall that in general, we may write $H_*$ as

$$H_* = \begin{pmatrix}
    h + h^{-1}|w|^2 & h^{-1} \bar{w} \\ h^{-1}w & h^{-1}
\end{pmatrix}.$$

For an approximate solution, we only require that it behaves nicely near all boundaries/corners, but allow it to behave awfully in the middle area. We only need to define $H_*$ near each boundary and use any arbitrary smooth one to fill the inside.\\

Note that a preferred compactification of $X$ and preferred coordinates near boundaries/corners are used as always. See appendix \ref{Appendix: Compactification of $X$} for the definitions of all types of boundaries/corners and local coordinates.

\paragraph{When $\rho$ is large (near type I boundary)}~\\

Recall that when $\rho$ is large, $\rho^2 = |z|^2 + y^2$. In this region, $$H_* = \begin{pmatrix}
    h + h^{-1}|w|^2 & h^{-1} \bar{w} \\ h^{-1} w & h^{-1}
\end{pmatrix}, $$ where  $h = e^{u_3}$, $w = \chi(\dfrac{|z|}{y}) \cdot \dfrac{R(z)}{Q(z)}$, and $u_3$ is the function constructed in section \ref{Section: The model solution} which satisfies

$$\Delta u_3 + e^{-2u_3} |P(z)|^2 = 0 $$
and generalized Nahm pole bounary condition plus real symmetry breaking condition.\\

Since $\rho$ is large, we may assume that $\chi(\dfrac{|z|}{y}) = 0$ at all roots of $Q(z)$ and $w$ well-defined. Moreover, $V(H, \varphi)$ can be represented by (see special case 1 in subsection \ref{Subsection: The metric representation}):

$$V(H, \varphi) = 
\dfrac{1}{2} \begin{pmatrix}
    E & \bar{F} \\ F & -E
\end{pmatrix},$$
with
$$ E  = \Delta u_3 + e^{-2u_3}(4|\bar{\partial} w|^2 + |\partial_y w|^2 + |P|^2) = e^{-2u_3}(4|\bar{\partial} w|^2 + |\partial_y w|^2),$$   $$F =
     e^{-u_3}(\Delta w - 2(\partial_y u_3)(\partial_y w) - 8 (\bar{\partial}w) (\partial u_3)). $$

Note that $\Delta w, \partial_y w$ and $\bar{\partial} w$ are only non-zero when $1 < \dfrac{|z|}{y} < 2$. So $E=F=0$ except in the region $1 < \dfrac{|z|}{y} < 2$. In this region, since we have assumed $\rho$ is large, $e^{-u_3} = O(e^{-y}) = O(e^{-\rho}),$ $| w| = O(\rho^{- \deg Q + \deg R})$, $|\nabla w| = O(\rho^{- \deg Q + \deg R} - 1)$, $|\nabla u_3| = O(1)$ and $|\Delta w| = O(\rho^{-2})$. So in fact, the first bullet of the extended Bogomolny equtauions

$$ |V(H, \varphi)| = O(e^{-\rho}),$$
as $\rho \rightarrow +\infty$, uniformly in $\dfrac{|z|}{y}$. 

\paragraph{When $r$ is small (near type III boundary)}~\\

We construction $H_*$ on a region such that $y$ is small and $|z|$ is bounded above. This region contains all the points with small $r$.\\

We have to work in a different basis. (That is to say, a different choice of $s_1, s_2$ in the definition of $(H, \varphi)$, see subsection \ref{Subsection: The metric representation}.)\\ 

Since $Q$ and $R$ are coprime, we assume $QS + TR = 1$, where $S, T$ are also polynomials. Consider the holomorphic $SL(2, \IC)$ gauge transformation $u = \begin{pmatrix}
    Q & T \\ -R & S
\end{pmatrix}$. (Recall, it is not an actual gauge transformation on the configuration, see subsection \ref{Subsection: The metric representation}.) It sends $\varphi$ to

$$u^{-1}\varphi u = \begin{pmatrix}
    S & -T \\ R & Q
\end{pmatrix} \begin{pmatrix}
    0 & 0 \\ P & 0
\end{pmatrix} \begin{pmatrix}
    Q & T \\ -R & S
\end{pmatrix} = \begin{pmatrix}
    -PQT & -PT^2 \\ PQ^2 & PQT
\end{pmatrix}.$$

In this basis, we choose 

$$u^*Hu = \begin{pmatrix}
    e^{u_3'} & 0 \\ 0 & e^{-u_3'}
\end{pmatrix}, $$
where $u_3'$ is the version of $u_3$ but using $PQ^2$ instead of $P$. That is to say, $u_3'$ satisfies the following equation:

$$\Delta u_3' + e^{-2u_3'}|P(z)Q(z)^2|^2 = 0. $$

Using the new pair $(H_u, \varphi_u) = (u^*Hu, u^{-1}\varphi u)$, its preferred configuration $\Psi_{H_u, \varphi_u}$ is $SU(2)$ gauge equivalent to $\Psi_{H, \varphi}$. (Recall the definition of $\Psi_{H, \varphi}$ is in subsection \ref{Subsection: The metric representation}.) This is in the special case 2 there. So the first bullet of the  extended Bogomolny equations is

$$V(H_u, \varphi_u) = \dfrac{1}{2} \begin{pmatrix}
    E & \bar{F}\\
    F & -E
\end{pmatrix} $$
with
$$E = \Delta u_3' + e^{-2u_3'}|PQ^2|^2 
 + e^{2u_3'}|PT^2|^2, ~~ F = 2e^{u_3'}(PT^2\bar{P}\bar{Q}\bar{T}) + e^{-u_3'} |PQ|^2 (T\bar{Q} + \bar{T}Q).$$

 When $|z|$ is bounded, all the polynomials $P, Q, R, T, S$ are bounded. Note that when $y$ is small and $|z|$ is bounded, $e^{u_3'} = O(1)$ and $e^{-u_3'}|PQ|^2 = O(\dfrac{1}{y})$. Thus

$$E = e^{2u_3'}|PT^2|^2 = O(1), ~~ F = O(\dfrac{1}{y}),$$
uniformly in $z$ when $|z|$ is bounded.\\

Going back to the original basis, $H$ is written as

$$H_* = (g^{-1})^* \begin{pmatrix}
    e^{u_3'} & 0 \\ 0 & e^{-u_3'}
\end{pmatrix} g^{-1} = \begin{pmatrix}
    \bar{S} & \bar{R} \\ -\bar{T} & \bar{Q}
\end{pmatrix} \begin{pmatrix}
    e^{u_3'} & 0 \\ 0 & e^{-u_3'}
\end{pmatrix} \begin{pmatrix}
    S & -T \\ R & Q
\end{pmatrix} $$
$$ = \begin{pmatrix}
    |S|^2 e^{u_3'} + e^{-u_3'} |R|^2 & - T \bar{S} e^{u_3'} + \bar{R}Q e^{-u_3'} \\
    - \bar{T} S e^{u_3'} + R \bar{Q}e^{-u_3'} & |T|^2 e^{u_3'} + |Q|^2 e^{-u_3'}
\end{pmatrix}. $$

We still have $|V(H_*, \varphi)| = O(\dfrac{1}{y})$ since its norm doesn't change under $SU(2)$ gauge transformations. This is what we want.

\paragraph{When $\psi$ is small (near type II boundary)}~\\

We assume $\psi$ is small. When $|z|$ is large at the same time, recall that $H_*$ is already defined  by $h = e^{u_3}$, $w = \dfrac{R(z)}{Q(z)}$. We call it 

$$H_1 = \begin{pmatrix}
    e^{u_3} + e^{-u_3}|R(z)|^2 |Q(z)|^{-2} & e^{-u_3} \bar{R}(z) \bar{Q}(z)^{-1} \\ e^{-u_3} R(z)Q(z)^{-1} & e^{-u_3}
\end{pmatrix}. $$

On the other hand, when  $|z|$ is small at the same time, $H_*$ is also defined to be

$$H_2 =  \begin{pmatrix}
    |S|^2 e^{u_3'} + e^{-u_3'} |R|^2 & - T \bar{S} e^{u_3'} + \bar{R}Q e^{-u_3'} \\
    - \bar{T} S e^{u_3'} + R \bar{Q}e^{-u_3'} & |T|^2 e^{u_3'} + |Q|^2 e^{-u_3'}
\end{pmatrix} = \begin{pmatrix}
    e^{\tilde{u}} + e^{-\tilde{u}}|\tilde{w}|^2 & e^{-\tilde{u}} \bar{\tilde{w}} \\ e^{-\tilde{u}} \tilde{w} & e^{- \tilde{u}}
\end{pmatrix}. $$

When $|z|$ is neither too large nor too small and $\psi$ (or $y$) is small, we have

\begin{itemize}
    \item $|S|, |T|$ are all bounded.
    \item $e^{u_3'} = O(y)$ (or $O(\psi)$) and $e^{u_3} = O(y)$ (or $O(\psi)$).
    \item As $y \rightarrow 0$, and when $|z|$ is not small (to stay away from zeros of $P$ and $Q$), $u_3 = - \ln y - \ln|P| + O(y)$,  while $u_3' = - \ln y - \ln (|PQ^2|) + O(y)$. So
    $$u_3' - u_3 = - 2 \ln |Q| + O(y).$$
\end{itemize}

All these properties imply that in this region, as $y \rightarrow 0$

$$\tilde{u} = u_3 + O(y), ~~\tilde{w} = \dfrac{R(z)}{Q(z)} +  O(y).$$

So using a smooth cut-off function in $z$ to connect $u_3$ and $\tilde{u}$, we get an $H_*$ defined on the entire region where $\psi$ is small such that $H_* = H_1$ when   with $|z|$ is large and $H_* = H_2$ when $|z|$ is small. The error in the middle is
$$V(H_*, \varphi) - V(H_1, \varphi)  = O(y^{-1}), ~~~ V(H_*, \varphi) - V(H_2, \varphi) = O(y^{-1}). $$

when $|z|$ is neither too small nor too big.\\

So $V(H_*, \varphi) = O(y^{-1})$ uniformaly in $z$ and $V(H_1, \varphi)= 0$ when $|z|$ is large.

\paragraph{To sum up}~\\

We have defined an $H_*$ near each boundary/corner with
$$\varphi = \begin{pmatrix}
    0 & 0 \\ P(z) & 0
\end{pmatrix}.$$

We can define $H_*$ in the middle area smoothly and arbitrarily. The properties of the pair $(H_*, \varphi)$ are:

\begin{itemize} 
    \item $H_*$ is smooth on $X$.
    \item The configuration that corresponds to it satisfies the generalized Nahm pole boundary condition as $y \rightarrow 0$ (but of the version of $PQ^2$, not $P$) and real symmetry breaking condition as $y \rightarrow +\infty$.
    \item When $y \rightarrow 0$, $V(H_*, \varphi) = O(\dfrac{1}{y})$ uniformly in $z$. 
    \item When $R \rightarrow +\infty$, $V(H_*, \varphi) = O(e^{-R})$ uniformly in $\dfrac{t}{|z|}$.
    \item When $|z|$ is large but $y$ is small, or when $y$ is large but $|z|$ is small, $V(H_*, \varphi) = 0$.
\end{itemize}

\paragraph{An estimate on $\Psi_{H_*, \varphi}$}~\\

For later analysis, we will need an estimate on the point-wise norm $|\Psi_{H_*, \varphi}|$. Here the norm is the sum of the norms on all its six components $A_1, A_2, A_y, \Phi_1, \Phi_2, \Phi_3$.\\

Recall that if $(H_*, \varphi)$ is written in the format 

$$H = \begin{pmatrix}
    h + h^{-1}|w|^2 & h^{-1} \bar{w} \\ h^{-1} w & h^{-1}
\end{pmatrix}, ~~ \varphi = \begin{pmatrix}
    0 & 0 \\ P(z) & 0
\end{pmatrix}, $$
then $\Psi_{H_*, \varphi}$ is written as
\begin{equation*} 
\left\{
\begin{array}{lr}
\Phi = \Phi_1 - i \Phi_2 = \begin{pmatrix}
   0 & 0 \\ h^{-1}P & 0
\end{pmatrix},&\\
\Phi_3 = \dfrac{i}{2h} \begin{pmatrix}
-\partial_y h & -\partial_y \bar{w} \\
-\partial_y w &  \partial_y h
\end{pmatrix}, &\\
A_y = \dfrac{1}{2h}\begin{pmatrix}
0 & \partial_y \bar{w} \\
- \partial_y w & 0
\end{pmatrix} , &\\
A_1 = \dfrac{1}{2h}\begin{pmatrix}
- i \partial_2h &  2\partial \bar{w}\\
- 2\bar{\partial} w & i \partial_2 h
\end{pmatrix},&\\
A_2 = \dfrac{i}{2h} \begin{pmatrix}
 \partial_1 h &   2\partial \bar{w} \\
 2\bar{\partial}w & - \partial_1 h
\end{pmatrix} .
\end{array}
\right.
\end{equation*} 

Examining the construction of $H$, it is clear that when $y \geq 1$, all these terms are point-wise bounded above by a constant which doesn't depend on $y$ or $z$.\\

When $y \rightarrow 0$, one also examines that $|\Psi_{H_*, \varphi}| = O(\dfrac{1}{y})$ uniformly in $z$. \\

So overall, there exists a constant $C > 0$ such that
$$ |\Psi_{H_*, \varphi}| \leq C (\dfrac{1}{y} + 1)$$
on the entire $X$.

\paragraph{A remark on geometrical meanings of $(P, Q, R)$ and boundary conditions}~\\

 We briefly explain the geometrical meanings of $(P, Q, R)$. There is no difference in this regard between our situation and the situation described by Dimakis in \cite{Dimakis2022TheField} without real symmetry breaking .\\

Note that zeros of $PQ^2$ are the ``knotted points" for the generalized Nahm pole boundary condition with degrees.  And zeros of $P$ are the zeros and vanishing orders of $\varphi$. Moreover, under a certain $SL(2, \IC)$ gauge, there is a ``small section" described by $\begin{pmatrix}
    R(z) \\Q(z)
\end{pmatrix}$. Readers may see \cite{Dimakis2022TheField} for more details. These data are irrelevant with the $SU(2)$ gauge transformations. So in particular, different choices of $(P, Q, R)$ do not give $SU(2)$ equivalent solutions.\\

\subsection{Improve the approximate solution when $y$ is small}\label{Subsection: Improve the approximate solution when $y$ is small}

Suppose $\Psi_* = \Psi_{H_*, \varphi}$ is the approximate solution in the last subsection. We modify ${\Psi}_*$ when $y$ is small and $|z|$ is bounded to make $V(\Psi_*) = O(y^n)$ for any positive integer $n$ as $y \rightarrow 0$. Without leading to an ambiguity, we may still use ${\Psi}_*$ to denote the modified version after this subsection in this paper.\\

The idea of the modification is to inductively construct a smooth section $s_k$ of $isu(2)$, $ k\geq 0$. Suppose $e^{s_0 + s_1 + \cdots + s_k}$ acts on $\Psi_{*}$ and gets $\Psi_k$. (Note that the summation is in the sense described in subsection \ref{Subsection: The deformation of a configuration} written in a sloppy way.) Then ${\Psi}_k $ has the property
$$V(\Psi_k) = O(y^{k 
- \epsilon}), ~~~ s_k = O(y^{k+1 - \epsilon}),$$
when $y \rightarrow 0$, with $\epsilon$ arbitrarily small. We assume all the derivatives of $s_k$ also have the correct vanishing order as $y \rightarrow 0$. We further assume each $s_k$ is supported in a region where $y$ is small and $|z|$ is bounded. So it doesn't ruin other good properties that $\Psi_*$ has as listed in \ref{Subsection: The approximate solution}. In particular, we assume originally $V(\Psi_*)$ = 0 when $y \leq \dfrac{1}{M}$ and $|z| \geq M$ for some large constant $M > 0$. Then for each $\Psi_k$, we inductively assume $V(\Psi_k) = 0$ when $y \leq \dfrac{1}{M}$ and $|z| \geq (2 - 2^{-(k+1)})M$.\\

Here is the construction:\\

Suppose $\Psi_{k-1}$ is constructed. (When $k = 0$, we assume $\Psi_{k-1} = \Psi_*$.) When $y$ is small, for any arbitrarily small $\epsilon > 0$,

$$V(\Psi_{k-1}) = O(y^{k-1 - \epsilon}), $$
and $V(\Psi_{k-1}) = 0$ when $y \leq \dfrac{1}{M}$  and $|z| \geq (2 - 2^{-k})M$ at the same time. \\

We use $L_{k-1}$ to denote the operator $L_{\Psi_{k-1}}$ for simplicity. We haven't studied the mapping properties of $L_{k-1}$ yet, which is in fact awful. But luckily since we don't care about the region when either $y$ or $|z|$ is large, we may modify $L_{k-1}$ to remain unchanged when both $y$ and $|z|$ are small, but equals $L_{\Psi_{k-1}'}$ when either $y$ or $|z|$ is large, where $\Psi_{k-1}'$ is a configuration with the same version of generalized Nahm pole boundary condition and without real symmetry breaking at $y \rightarrow +\infty$. When $y$ is small and $|z|$ is bounded, as an elliptic operator with ``edge" type (strictly speaking, $y^2 L_{\Psi_{k-1}'}$ is the elliptic operator of ``edge" type locally), only non-leading terms are modified. This operator $L_{\Psi_{k-1}'}$ is already well studied in \cite{Mazzeo2017TheKnots} and \cite{He2020TheII}. And it equals $L_{k-1}$ when $y$ is small and $|z|$ is bounded (whose bound is larger than $2M$, say when $|z| \leq 100M$).\\

Suppose $\chi_1(y)$ is a smooth cut-off function in $y$ which is $1$ when $y$ is small and $0$ when $y$ is large. Suppose $\chi_2(z)$ is a smooth cut-off function in $z$ which is $1$ when $|z| \leq 100M$ and $0$ when $|z|$ is large.   By the same reason as the proof of proposition 7.1 in \cite{He2020TheII} or the arguments in paragraph 4.3 in \cite{Dimakis2022TheField}, we have a smooth section $\tilde{s}_k$ such that when $y \rightarrow 0$, in this region,

$$L_{\Psi_{k-1}'}(\tilde{s}_k) = -\chi_1(y)\chi_2(z)V(\Psi_{k-1}), ~~~ \tilde{s}_k = O(y^{k+1 - \epsilon}). $$

All the derivatives of $\tilde{s}_k$ also have the appropriate decay rates at $y \rightarrow 0$. Note that $\tilde{s}_k$ may behave bad outside the region mentioned above.\\

In fact, the operator $L_{\Psi_{k-1}'}$ is locally invertible in the region $y$ is small and $|z| \leq 100M$ above the lower Fredhohm weight. To be more precise, if $V$ is supported in this region and $V \in y^{\delta}C^{k, \alpha}_{\text{ie}}$ when $\delta > -2$, then there is an $\tilde{s}_k \in y^{\delta + 2 - \epsilon}C^{k, \alpha}_{\text{ie}}$ such that $L_{\Psi_{k-1}'} \tilde{s}_k = V$ in this region. The definition of $C^{k, \alpha}_{\text{ie}}$ can be found in \cite{He2020TheII}, \cite{Dimakis2022TheField} or appendix \ref{Appendix: Weighted Holder spaces of (iterated) edge type}. The fact that the Fredholm weight is above $-2$ can be found in \cite{Mazzeo2017TheKnots}, \cite{He2020TheII} or \cite{Dimakis2022TheField}. (Strictly speaking, near a corner both $\psi$ and $r$ goes to $0$, we should use $\psi^{\delta}r^{\delta} C^{k, \alpha}_{\text{ie}}$. However even there, $\psi r$ is equivalent with $y$.) Since we only require the equation holds in the local region mentioned above, this is a local property and there is no need to define the space in other parts of $X$, say, when $R \rightarrow +\infty$. This also means, $\tilde{s}_k$ is allowed to behave bad outside the region that we are interested in. Note that the small constant $\epsilon$ exists because of the fact that $\tilde{s}_k$ may have additional terms of $O(y^{\delta + 2}\ln y)$ which make it not reside exactly in  $y^{\delta + 2}C^{k, \alpha}_{\text{ie}}$.\\

We may take a third cut-off function $\chi_3(z)$ in $z$ which is $1$ when $|z| \leq (2 - 2^{-k}) M$ and $0$ when $|z| \geq (2 - 2^{-(k+1)})M$. Recall the fact that $V(\Psi_{k-1}) = 0$ when $|z| \geq (2 - 2^{-k}) M$ and $y$ is small. This implies that $L_{\Psi_{k-1}'}(\tilde{s}_k) = 0$ in the same region.\\

Thus let $s_k = \chi_3(z) \chi_1(y) \tilde{s}_k$. We have $s_k$ is supported in the region $|z| \leq (2 - 2^{k+1})M$ and $y$ is small. And  in particular, when $|z| \leq (2 - 2^{k+1})M$ and $y \leq \dfrac{1}{M}$ (We assume $\chi_1(y) = 1$ in this region),

$$L_{k-1}(s_k) = L_{\Psi_{k-1}'} (s_k) = L_{\Psi_{k-1}'}(\chi_3(z)) \tilde{s}_k. $$

So let $\Psi_{k}$ be the configuration that we get after applying $e^{e_k}$ on $\Psi_{k-1}$. According to the first fact in subsection \ref{Subsection: The deformation of a configuration},

$$V(\Psi_k) = V(\Psi_{k-1}) + L_{k-1}(s_k) + Q(s_k). $$

When $|z| \leq (2 - 2^k)M$ and $y \leq \dfrac{1}{M}$, 

$$V(\Psi_k) = V(\Psi_{k-1}) + L_{k-1}(s_k) + Q(s_k) = Q(s_k) = O(y^{k-\epsilon}). $$

When $(2 - 2^k)M \leq |z| \leq (2 - 2^{k+1})M$ and $y \leq \dfrac{1}{M}$, $$V(\Psi_{k-1}) = L_{k-1}(\tilde{s}_k) = 0$$ and
$$L_{k-1}(\chi_3(z)\tilde{s}_k)  \lesssim |\nabla^2 \chi_3|  |\tilde{s}_k| + |\nabla \chi_3| |\nabla \tilde{s}_k| =  O(y^{k-\epsilon}). $$
$$V(\Psi_k) = V(\Psi_{k-1}) + L_{k-1}(s_k) + Q(s_k) = L_{k-1}(\chi_3(z)\tilde{s}_k) + Q(s_k) = O(y^{k-\epsilon}). $$
Finally, when $|z| \geq (2 - 2^{k+1})M$ and $y \leq \dfrac{1}{M}$, we have $s_k = 0$ thus $\Psi_k = \Psi_{k-1}$. So the inductive construction is finished.\\

We take summation in a convergent way of $\{s_k\}$ (and in the sense described in subsection \ref{Subsection: The deformation of a configuration}) and get a new configuration, still denoted as $\Psi_*$, with the additional property besides what are listed in subsection \ref{Subsection: The approximate solution}: 

$$V(\Psi_*) = O(y^n), $$
for any positive integer $n$ when $y \rightarrow 0$ uniformly in $z$. Here ``in a convergent way" means we may freely multiply a cut-off function in $y$ supported and equals to $1$ near $y = 0$ (but becomes $0$ quickly when $y$ gets away from $0$) onto each $s_k$ (whose support depends on $s_k$) to make the infinite sum of $s_k$ converge. In the remaining of this paper, we will use this new $\Psi_*$ instead of the old one.

\subsection{The continuity argument}\label{Subsection: The continuity argument}

Let ${\Psi}_{*}$ be the approximate solution from the last subsection.  Let $s_{*} = V(\Psi_*)$.  Clearly $s_* \in isu(2)$. Consider $e^{\frac{s_*}{2}}$ acting on $\Psi_{*}$ and call it $\Psi_0$. Note that $s_*$ vanishes up to any finite order on any boundaries. (We'll need this property to make sure that $\Psi_0$ still satisfies the same generalized Nahm pole condition and real symmetry condition as $\Psi_*$. ) Then $e^{-\frac{s_*}{2}}$ acting on $\Psi_0$ gets back to $\Psi_{*}$. We have
$$V(\Psi_0, -s_*) = s_*, $$
where the notation $V(\Psi_0, -s_*)$ is defined in subsection \ref{Subsection: The deformation of a configuration}.\\

On the other hand, by the first fact described in subsection \ref{Subsection: The deformation of a configuration}
$$s_*  = V(\Psi_0) - L_{\Psi_0}(s_*) + Q(-s_*). $$

Thus $V(\Psi_0) = 0 $ whenever $s_* = 0$. And $V(\Psi_0) = O(y^n)$ for any $n$ when $y \rightarrow 0$ just like $s_*$. And it decays exponentially when $R \rightarrow 0$ just like $s_*$. That is to say, for some small constant $\epsilon > 0$ ,$V(\Psi_0) = O(e^{- \epsilon R})$ as $R \rightarrow +\infty$.\\

The following equation
$$V(\Psi_0, s) + ts = 0 $$
has an obvious solution $s = -s_*$ when $t = 1$. In general, suppose $s$ is a section of $isu(2)$. Let $S\subseteq [0, 1]$ be the set of $t$ values such that
$$V(\Psi_0, s) +ts = 0 $$
has a solution with $s \in \mathcal{X}^{k}_{1/2 - \epsilon, 1/2 - \epsilon, -1 + \epsilon}$ for some small $\epsilon > 0$ and any non-negative integer $k$, where the definition of $\mathcal{X}^{k}_{1/2 - \epsilon, 1/2 - \epsilon, -1 + \epsilon}$ can be found in appendix \ref{Appendix: Weighted Holder spaces of (iterated) edge type}. Then $S$ is non-empty. (Because $1 \in S$.)\\

By way of looking ahead, we prove that $S = [0, 1]$ by showing it is both relatively open and relatively closed. In particular, when $t = 0$, it gives an actual solution to the extended Bogomolny equations. In fact, we have the following five facts:

\begin{itemize}
    \item The initial $s_1 = -s_*$ is in $\mathcal{X}^{k}_{\frac{1}{2}, \frac{1}{2}, -l}$ for any non-negative integers $k, l$. This is straightforward from the construction.
    \item Suppose at some $t = t_{\diamond} > 0$, there is a solution $s_{t_{\diamond}} \in \mathcal{X}^{k}_{\frac{1}{2}, \frac{1}{2}, -l}$ for all $k, l$. Then for some small $\epsilon > 0$ depending on $t_{\diamond}$, there is also a solution in $\mathcal{X}^{k}_{\frac{1}{2}, \frac{1}{2}, -l}$ when $t_{\diamond} - \epsilon < t < t_{\diamond} + \epsilon$. This is proved in subsection \ref{Subsection: More regularities}.
    \item For each $0 < t_0 \leq 1$, suppose when $t_0 \leq t \leq 1$ there is a solution $s_t \in \mathcal{X}^{k}_{\frac{1}{2}, \frac{1}{2}, -l}$. Then there is a constant $C$ which only depends on $k, l, t_0$ and $\Psi_0$ such that
$$||s_t||_{\mathcal{X}^{k}_{\frac{1}{2}, \frac{1}{2}, -l}} \leq C. $$
Note that the constant $C$ doesn't depend on $s_t$ or $t$. This is proved in subsection \ref{Subsection: A priori estimates}.
\item Suppose $t > 0$. If $s_t$ is a solution in $\mathcal{X}^k_{\frac{1}{2}- \epsilon, \frac{1}{2}-\epsilon, -l}$ for all $k$ and $l$ and some small $\epsilon > 0$. Then it is also in $\mathcal{X}^k_{\frac{1}{2}, \frac{1}{2}, -l}$ for all $k$ and $l$. This is proved in subsection \ref{Subsection: More regularities}.
\item  There is a constant $C$ which depends only on $k$ and $\Psi_0$, such that for any $0 \leq t \leq 1$, if there is a solution $s_t \in \mathcal{X}^k_{\frac{1}{2}, \frac{1}{2}, -1}$. Then
$$||s_t||_{\mathcal{X}^k_{\frac{1}{2}, \frac{1}{2}, -1}} \leq C.$$
Note that $C$ doesn't depend on $s_t$ or $t$. This is proved in subsection \ref{Subsection: A priori estimates}.
\end{itemize}

\paragraph{Proof that $S = [0, 1]$ using the above facts}~\\

Let $S'$ be the subset of $S$ such that additionally $s_t \in \mathcal{X}^k_{\frac{1}{2}, \frac{1}{2}, -l}$ for all $k, l$. Then $1 \in S'$.\\

The second bullet indicates that $S'$ is relatively open in $[0, 1]$. If there is a sequence $\{t_n\}$ in $S'$ converging to $ t_0 > 0$. Then by the third bullet and by choosing a sub-sequence, we may assume $s_{t_n}$ converges to an element in $\mathcal{X}^k_{\frac{1}{2}-\epsilon, \frac{1}{2}-\epsilon, -l}$ for all $k, l$. We don't need to add an $\epsilon$ on $-l$ since it is true for all $l$. Then by the fourth bullet the limit $s_{t_0}$ actually lies in $\mathcal{X}^k_{\frac{1}{2}, \frac{1}{2}, -l}$. This limit solves the $t_0$ version of the equation. So $S'\cap (0, 1]$ is a relatively closed subset of $(0, 1]$. Thus $(0, 1] \subset S'$. In particular, $(0, 1] \subset S$.\\

Finally, choose a sequence $t_n \rightarrow 0$. According the the fifth bullet, by choosing a sub-sequence, we may assume $s_{t_n}$ converges to some $s_0 \in \mathcal{X}^k_{\frac{1}{2} -\epsilon, \frac{1}{2}- \epsilon, -1 + \epsilon}$. This limit $s_0$ solves the $t = 0$ version of the equation. Thus $0 \in S$.

\subsection{A priori estimates}\label{Subsection: A priori estimates}

This subsection derives some a priori estimates, which prove the third and the fifth bullets of the facts listed in the end of subsection \ref{Subsection: The continuity argument}. These estimates are largely due to Dimakis' estimates in \cite{Dimakis2022TheField}. A large portion of them are also nearly identical to Mazzeo and He's work in \cite{He2019TheCondition} and Jacob and Walpuski's work in \cite{Jacob2019HermitianManifolds}. Moreover, the local $C^{0, \alpha}$ estimate may originate in Bando and Siu's paper \cite{BANDO1994STABLEMETRICS} and the local $C^{k, \alpha}$ estimate may originate in Hildebrandt's work in \cite{Hildebrandt1985HarmonicManifolds}. The author is not attempting to trace all the origins here, but wants to give the readers the impression that these estimates are all standard in analysis in some sense.

\paragraph{Local $L^{\infty}$ estimate}~\\

Suppose we have a solution $$V(\Psi_0, s_t) + ts_{t} = 0.$$
 Then according to the second fact (WeitzenBock formula) described in subsection \ref{Subsection: The deformation of a configuration}, we have
$$- t|s_t|^2 = <V(\Psi_0, s_t), s_t>  =  <V(\Psi_0), s_t> - \dfrac{1}{2}\Delta (|s_t|^2) + \sum\limits_{i=1}^3|v(-s)D_i^*s|^2. $$

\begin{lemma} (Dimakis' estimate in \cite{Dimakis2022TheField})\label{lemma of C infinity estimate 1}
    For each $k \geq 1$, there exists a constant $C_1$ which only depends on $V(\Psi)$, $k$ and $t_0 > 0$, such that for each $t \in (t_0, 1]$, if $s_t \in \mathcal{X}^k_{\frac{1}{2}, \frac{1}{2}, -l}$ for all $k, l$. Then
    $$\sup|\rho^k s_t| \leq C_1. $$
    Moreover, when $k = 1$, if we only assume $t \in [0, 1]$ and $s_t \in \mathcal{X}^k_{\frac{1}{2}, \frac{1}{2}, -1}$ for all $k$, then there is a constant $C$ which doesn't depend on $t$ such that
    $$\sup|\rho s_t| \leq C. $$
Note that the definition of $\rho$ can be found in appendix \ref{Appendix: Compactification of $X$}.
\end{lemma}

\begin{proof}
The constants $C_1$ and $C$ may change from line to line. Keep in mind that $C_1$ only depends on $\Psi_0, k$ and $t_0$, while $C$ only depends on $\Psi_0$.\\

Note that originally we have

$$|s_t| \leq \rho^{-k}
C_k(s_t),$$
where $C_k(s_t)$ depends on $s_t$.\\

Point-wise we have
$$- \Delta(|s_t|^2) + 2t|s_t|^2 ~ \leq ~ 2|V(\Psi_0)| \cdot |s_t|$$

Consider the Green's function for $-\Delta + 2t$ on $X$. This Green's function is

$$G_{2t, q}(p) = \dfrac{e^{-\sqrt{2t}|p - q|}}{4\pi|p - q|} - \dfrac{e^{-\sqrt{2t}|p - \bar{q}|}}{4\pi|p - \bar{q}|}, $$
where $\bar{q}$ is the reflection of $q$ about $y = 0$ plane.\\

First we assume $t \geq t_0 > 0$ and prove the first bullet. Choose a ball $B$ of radius $\dfrac{\rho}{2}$ centered at $q$. The maximal principle implies that

$$|s_t|^2 \leq \int_X 2 G_{2t, q}|V(\Psi_0)| |s_t| = \int_{B\cap X}2 G_{2t, q}|V(\Psi_0)| |s_t| + \int_{X \backslash B}2 G_{2t, q}|V(\Psi_0)| |s_t| $$
$$\leq 2^{k+1} C_k(s_t) (\int_{B\cap X} {\rho}^{-k} G_{2t, q}|V(\Psi_0)| + \int_{X \backslash B} G_{2t, q}|V(\Psi_0)|). $$

Note that $|V(\Psi_0)|$ has exponential decay as $\rho \rightarrow +\infty$ and $G_{2t, q}$ has exponential decay with respect to the distance to $q$ outside of the ball $B$. So 
$$(\int_{B \cap X} {\rho}^{-k} G_{2t, q}|V(\Psi_0)| + \int_{X \backslash B} G_{2t, q}|V(\Psi_0)|) \leq C_1y^{-k}. $$
This implies
$$|s_t| \leq \sqrt{C_1 C_k(s_t)}{\rho}^{-k}. $$
Repeating this finitely many times we get,
$$|s_t| \leq C_1{\rho}^{-k} $$

Next we assume $t \in [0, 1]$ and prove the second bullet.\\

We simply have

$$- \Delta(|s_t|^2)  ~ \leq ~ 2|V(\Psi_0)| \cdot |s_t|. $$

Originally we have
$$|s_t| \leq C(s_t) \rho^{-1} \leq C(s_t).$$

And the same maximal principle applies and we have

$$|s_t|^2 \leq \int_X 2 G_{0, q}|V(\Psi_0)| |s_t| \leq 2 C(s_t) \int_X |G_{0, q}| |V(\Psi_0)|$$ $$ = 2 C(s_t) (\int_{X\cap B} |G_{0, q}| |V(\Psi_0)| + \int_{X \backslash B}|G_{0, q}| |V(\Psi_0)|).$$

Note that  $\displaystyle \int_{X \backslash B} G_{0, q}|V(\Psi_0)|$ is bounded above by $C \rho^{-2}$ since $|V(\Psi_0)|$ has an exponential decay in $\rho$ when $\rho \rightarrow +\infty$.\\

On the other hand, note that

$$|G_{0, q}(p)| = \dfrac{1}{4\pi} \dfrac{4y(p)y(q)}{|p - q||p - \bar{q}| (|p - q| + |p - \bar{q}|)}, $$
where $y(p), y(q)$ means the $y$ coordinates of $p$ and $q$ respectively. On $X \backslash B$,
$$\int_{X \backslash B} G_{0, q}|V(\Psi_0)|\leq C \int_{X \backslash B} \dfrac{y(p)y(q)}{|p - q|^3} |V(\Psi_0)| dp \leq C\rho^{-2}\int_{X} y|V(\Psi_0)| \leq C\rho^{-2}. $$

So after all

$$|s_t|^2 \leq C C(s_t) \rho^{-2}. $$
Which means
$$|s_t| \leq \sqrt{C C(s_t)}  \rho^{-1}.$$
Repeating this finitely many times we get (remember that $C$ may change from line to line)
$$|s_t| \leq C \rho^{-1}. $$

\end{proof}

We also need a point-wise estimate of $|s_t|$ when $y \rightarrow 0$. We use $\hat{y}$ to represent $y$ when $y$ is small and $1$ when $y$ is not small.

\begin{lemma} (This is inspired by a similar Taubes' estimate in  \cite{Taubes2021TheR})\label{lemma of C infinity estimate 2}

There is a constant $C$ which only depends $V(\Psi_0)$, such that, 
$$|s_t| \leq C y^{\frac{1}{2}}. $$
\end{lemma}

\begin{proof}

Since by lemma \ref{lemma of C infinity estimate 1} $|s_t|$ is bounded above by $C$ on the entire $X$. We only need to assume $y$ is small. Recall in the proof of lemma \ref{lemma of C infinity estimate 1}, the maximal principle that we have:

$$|s_t|^2 \leq \int_X 2 G_{0, q}|V(\Psi_0)| |s_t|. $$

Suppose initially  we suppose 
$$|s_t| \leq \min\{C(s_t) y^{\frac{1}{2}}, C\}, $$
where $C(s_t)$ depends on $s_t$.\\

Let $B$ be the ball of radius $\dfrac{y(q)}{2}$ centered at $q$. Then $|G_{0, q}| \leq C y^{-2} y(q) $ outside of this ball and 
$$ \int_{X \backslash B} 2 G_{0, q}|V(\Psi_0)| |s_t| \leq CC(s_t)y(q) \int_{X \backslash B} y^{-2}  |V(\Psi_0)|.$$

Note that $\displaystyle \int_{X} y^{-2}  |V(\Psi_0)|$ is bounded. So centered at the point $q$, we have

$$ \int_{X \backslash B} 2 G_{0, q}|V(\Psi_0)| |s_t| \leq CC(s_t)y(q).$$

On the other hand, within the ball $B$, since $|V(\Psi_0)| = O(y^n)$ for any positive integer $n$ when $y$ is small, we have
$$\int_B 2 G_{0, q}|V(\Psi_0)| |s_t|  \leq C C(s_t) y. $$
So after all  
$$|s_t|^2 \leq CC(s_t) y. $$

Repeating this finitely many times we get
$$|s_t| \leq C \sqrt{y}. $$

\end{proof}

\paragraph{The interior $\rho^{-1}C^{0, \alpha}$  estimates for $s_t$ away from $y = 0$ } (Bando and Siu's estimates)~\\

We continue to assume $s_k \in \mathcal{X}^k_{\frac{1}{2}, \frac{1}{2}, -l}$ solves

$$V(\Psi_0, s_t) + ts_{t} = 0.$$

We first work in a region away from the $y = 0$ boundary.

\begin{lemma}\label{lemma of C(0, alpha) estimate 1}
    Suppose $B_r$ is a ball of radius $r$ in $X$ which is far away from the $y = 0$ boundary. We assume $r \leq 1$. Let $G$ be the Green's function of $\Delta$ centered at the center of the Ball (basically $4\pi$ times $1$ over the distance to the center). Let $$f(r): = \int_{B_r}G|\nabla s_t|^2.$$
    Then there exists a constant $C$ and $\alpha \in (0, 1)$ which only depends on $\Psi_0$, such that
    $$f(r) \leq C r^{2\alpha}. $$

\end{lemma}

\begin{proof}

Recall the fact that $|s_t|$ is bounded above by a constant that only depends on $V(\Psi_0)$. So the point-wise norm (largest eigenvalue) of the operator $ad_{-s_t}$ is also bounded uniformly on the entire $X$. Note that $ad_{s_t}$ is a self-dual map. So all its eigen-values are real number. Thus $$|v(-s)| = |\sqrt{\dfrac{e^{ad_{-s}} - 1}{ad_{-s}}}|  $$
is bounded below uniformly (that is to say, the norm of all its eigenvalues are bounded below by a positive constant).\\

So for a constant $C$ which only depends on $V(\Psi_0)$, but may change from line by line, we have
$$-\Delta(|s_t|^2) + \dfrac{1}{C}\sum\limits_{i = 1}^3|D_i^*s_t|^2 + t|s_t|^2 \leq |V(\Psi_0)||s_t|. $$

From now on, let $\chi_r$ be a standard smooth cut-off function which is $1$ on $B_r$ and $0$ outside of $B_{2r}$. We assume $|\nabla^k \chi_r| \leq \dfrac{C}{r^k}$ uniformly.\\

\textbf{Step 1} We first show that $f(r)$ is uniformly bounded above. We have
$$-\Delta(|s_t|^2) + \dfrac{1}{C}|\nabla s_t|^2 \leq |V(\Psi_0)||s_t| \leq C. $$ 

On a ball $B_r$ of radius $r$ centered at $p$, 
$$\int_{B_{r}} G |\nabla s_t|^2 \leq \int_{B_{2r}} \chi_r G |\nabla s_t|^2 \leq  C\int_{B_{2r}} \chi_{r} G (\Delta (|s_t|^2) + C)$$
$$\leq  C|s_t|^2(p)  + 
 C\int_{B_{2r}}|s_t|^2 ((\Delta \chi_r)G + (\nabla \chi_r \nabla G)) + C \leq C.  $$

 \textbf{Step 2} Let $\bar{s}_t$ be the average of $s_t$ on the region $B_{2r} \backslash B_r$, where $B_{2r}$ and $B_r$ have the same center. That is to say,
 $$\bar{s}_t = \dfrac{1}{\text{Vol}(B_{2r} \backslash B_r)}\int_{B_{2r} \backslash B_r} s_t. $$

 Since $\bar{s}_t$ is a constant, all its derivatives are $0$. Moreover, since we are working in a region away from the $y = 0$ boundary, we know that $|\Psi_0|$ is uniformly bounded and hence
 $$|D_1\bar{s}_t|, ~ |D_2\bar{s_t}|, ~ |D_3\bar{s}_t|, ~ |D_1(\gamma(-\bar{s}_t))|, ~|D_2(\gamma(-\bar{s}_t))|, ~|D_3(\gamma(-\bar{s}_t))| $$
 are all bounded uniformly. Thus
 $$|V(\Psi_0, \bar{s}_t)| \leq |V(\Psi_0)| + |L_{\Psi_0} \bar{s}_t | + |Q(\bar{s}_t)| \leq C.$$
 Suppose $e^{2s_t} =  e^{\bar{s}_t}e^{2\sigma}e^{\bar{s}_t}$. (See subsection \ref{Subsection: The deformation of a configuration} for details on how to add Hermitian gauge transformations.) Then $\sigma \in isu(2)$. We will verity that
$$|\sigma| \leq C|s_t - \bar{s}_t|, ~~~ |\nabla s| \leq C |\nabla \sigma|. $$

In fact, for the first inequality, suppose the eigen-values of $\sigma$ are $\lambda_1$ and $\lambda_2$. Then $\tr(e^{\sigma}) = e^{\lambda_1} + e^{\lambda_2}$. And  the norm $|\sigma|$ is equivalent with $|\lambda_1| + |\lambda_2|$, which is also equivalent with $\tr(e^{\sigma}) - 2$ when $|\sigma|$ is bounded. On the other hand, we indeed have $|s_t|, |\bar{s}_t|, |\sigma|$ are all bounded and
$$\tr(e^{2\sigma}) = \tr(e^{ - \bar{s}_t}e^{2s_t}e^{-\bar{s}_t}) = \tr(e^{2s_t}e^{-2\bar{s}_t}) = \tr (e^{2(s_t - \bar{s}_t)}) .  $$
Then the first inequality follows.\\

The second inequality is because
$$\gamma(2s_t) (2\nabla(s_t)) e^{2s_t} = e^{\bar{s}_t} \gamma(2\sigma) (2\nabla \sigma) e^{2\sigma} e^{\bar{s}_t}.$$

Again, $|\gamma(2s_t)|, |e^{2s_t}|, |e^{2\sigma}|, |\gamma(2\sigma)|, |e^{\bar{s}_t}|$ are all bounded above and below away $0$. So $|\nabla s_t|$ and $|\nabla \sigma|$ can bound each other.\\

Applying the Weitzenbock formula in subsection \ref{Subsection: The deformation of a configuration} and the fact that $s_t, \bar{s}_t, \sigma$ are bounded above, we get
$$ - \dfrac{1}{2} \Delta (|\sigma|^2) + \dfrac{1}{C} |\nabla \sigma|^2 \leq ~~ (|V(\Psi_0, s_t))| + |V(\Psi_0, \bar{s}_t)|) |\sigma| \leq  ~ C. $$

So we have

$$\int_{B_r} G|\nabla s|^2 \leq C \int_{B_r}G|\nabla \sigma|^2 \leq C\int_{B_{2r}} \chi_r G|\nabla \sigma|^2 \leq C \int_{B_{2r}} \chi_r G(1 + \Delta(|\sigma|^2)) $$
$$\leq Cr^2 + C\int_{B_{2r}}(|\Delta(\chi_r) G| + |\nabla \chi_r| |\nabla G|)|\sigma|^2 \leq Cr^2 + C \dfrac{1}{r^3} \int_{B_{2r} \backslash B_r} |\sigma|^2 \leq Cr^2 + C\int_{B_{2r} \backslash B_r}G |\nabla s_t|^2.$$

The last inequality above is the Poincare inequality

$$\int_{B_{2r}\backslash B_r}|\sigma|^2 \leq C\int_{B_{2r} \backslash B_r} |s_t - \bar{s}_t|^2 \leq C r^2 \int_{B_{2r} \backslash B_r} |\nabla s_t|^2.$$

So $f(r) \leq Cr^2 + C(f(2r) - f(r)).$\\

By some algebraic manipulations and the fact that $f(r)$ is an increasing function, one sees from above that
$$f(r) \leq C r^{2\alpha}, $$
for some $\alpha \in (0, 1)$.
\end{proof}

Note that the proof of lemma \ref{lemma of C infinity estimate 1} implies that locally
$|\nabla s_t|$ has a finite $L^{2,1 + 2\alpha}$ norm (the Morry-Campanato norm, see appendix \ref{Appendix: Morrey-Camponato spaces and inequalities}). By theorem \ref{theorem D2} in the same appendix, we know that 
$$[s_t]_{C^{0, \alpha}(B)} \leq C.$$
So away from the $y = 0$ boundary (say, $y \geq 1$),
$$|s_t|_{C^{0, \alpha}} \leq C. $$

Moreover, when $\rho$ is large, (Here $\rho$ means the $\rho$ value of the center of the ball, hence is a constant.)

$$-\Delta(|\rho s_t|^2) + \dfrac{1}{C}\sum\limits_{i = 1}^3|\rho D_i^*s_t|^2 + t|\rho^k s_t|^2 \leq |V(\Psi_0)|\rho^{2}|s_t|, $$
which is bounded above uniformly. And the proof doesn't need to change if we replace $s_t$ with $\rho s_t$, where $\rho$ is the $\rho$-value of the center of the ball $B$. In particular, we still have
$$|V(\Psi_0, \rho \bar{s}_t )| \leq C. $$
(We need to use the first fact in subsection \ref{Subsection: The deformation of a configuration} and lemma \ref{lemma of C infinity estimate 2} here.)\\

So in fact, we get when $y$ is not small (say $y \geq 1$), we get

 $$||s_t||_{\rho^{-1} C^{0, \alpha}} \leq C$$
 for some $\alpha \in (0, 1)$.\\

 Moreover, by exactly the same reason, if $t \geq t_0 > 0$ and allow $C$ to depend on $t_0$ and an additional positive integer $l$. Then

 $$||s_t||_{\rho^{-l} C^{0, \alpha}} \leq C$$
 for some $\alpha \in (0, 1)$.\\

\paragraph{The $y^{\frac{1}{2}}C_{\text{\text{ie}}}^{0, \alpha}$  estimates for $s_t$ near $y = 0$ } (This is an ``adapted to the edge boundary version" of Bando and Siu's estimate, being analog with an estimate of He and Mazzeo introduced in \cite{He2020TheII}.)~\\

We cannot use the same argument to show that $s_t$ is still locally in $C^{0, \alpha}$ near $y = 0$ since we don't have a uniform bound on $|\Psi_0|$ near $y = 0$, which is required in the estimate of $|V(\Psi_0, \bar{s}_t)|$. In fact, we only have
$|\Psi_0| \leq \dfrac{C}{y}$ when $y$ is small. On the other hand, the definition of $C^{0, \alpha}_{\text{ie}}$ is also adjusted near this boundary in a scale-invariant way. So in fact, we need to do the argument in a scale-invariant way near the boundary.\\

Consider a ball $B$ of radius $\dfrac{y(q)}{4}$ centered at a point $q$. We hope get an $y^{\frac{1}{2} } C^{0, \alpha}_{\text{ie}}$ estimate for $s_t$ and some small $\epsilon > 0$.\\

Thanks to the fact that $|s_t| = O(\sqrt{y})$ and the fact that $|V(\Psi_0)| = O(y^{n})$ for any positive integer $n$ when $y$ is small.\\

We still define

$$f(r) = \int_{B_r} G|\nabla s_t|^2, $$
where $B_r$ is the ball of radius $r$ centered at $q$. We assume $0 < r \leq \dfrac{y(q)}{4}$.\\

In the ball, we have
$$-\Delta(|s_t|^2) + \dfrac{1}{C}\sum\limits_{i = 1}^3|D_i^*s_t|^2 + t|s_t|^2 \leq |V(\Psi_0)||s_t| \leq Cy(q). $$

Then by the same reason as in the step 1 of the proof of lemma \ref{lemma of C(0, alpha) estimate 1}, 
$$f(r) \leq C|s_t|^2(q) + C\int_{B_{2r}} |s_t|^2 (|\Delta \chi_r| |G| + |\nabla \chi_r||\nabla G|) + Cy(q) \leq Cy(q). $$

On the other hand, in step 2, we have $|\Psi_0| = O(y^{-1})$ and $|s_t| \leq C \sqrt{y}$. So
 $$|D_1\bar{s}_t|, ~ |D_2\bar{s_t}|, ~ |D_3\bar{s}_t|, ~ |D_1(\gamma(-\bar{s}_t))|, ~|D_2(\gamma(-\bar{s}_t))|, ~|D_3(\gamma(-\bar{s}_t))| $$
 are all bounded by $C y^{-\frac{1}{2}}$ uniformly. So $|L_{\Psi_{0}} \bar{s}_t| \leq Cy^{-\frac{1}{2}}$ and $|Q(\bar{s}_t)| \leq Cy^{-1}$, where $\bar{s}_t$ has the same meaning as in the proof of lemma \ref{lemma of C(0, alpha) estimate 1}.\\

So what we get is 
$$|V(\Psi_0, \bar{s}_t)| \leq Cy(q)^{-1}. $$

And like in step 2 of the proof of lemma \ref{lemma of C(0, alpha) estimate 1},

$$- \dfrac{1}{2}\Delta(|\sigma|^2) + \dfrac{1}{C}|\nabla \sigma|^2 \leq Cy(q)^{-1} \cdot y(q)^{\frac{1}{2}} = C y(q)^{-\frac{1}{2}}. $$

So what we have is

$$\int_{B_r} G|\nabla s|^2 \leq C \int_{B_{2r}} \chi_r G(y(q)^{-\frac{1}{2}} + \Delta(|\sigma|^2)) $$
$$\leq Cy(q)^{-\frac{1}{2}} r^2 + C \dfrac{1}{r^3} \int_{B_{2r} \backslash B_r} |\sigma|^2 \leq Cy(q)^{-\frac{1}{2}}r^2 + C\int_{B_{2r} \backslash B_r}G |\nabla s_t|^2.$$

In fact, what we get is

$$f(r) \leq Cy(q)^{-\frac{1}{2}}r^2 + C(f(2r) - f(r)). $$

Doing algebraic manipulations and keep in mind that $r \leq \dfrac{y(q)}{4}$ and we have $f(r) \leq Cy(q)$ for any $0 < r \leq \dfrac{y(q)}{4}$, we get

$$f(r) \leq Cy(q)^{1 - 2\alpha} r^{2\alpha} $$
for some $\alpha \in (0, 1)$.\\

This implies that

$$\int_{B_r}|\nabla s_t|^2 \leq Cy(q)^{1-2 \alpha} r^{2\alpha + 1}. $$

So

$$[s_t]_{C^{0, \alpha}} \leq C y(q)^{1 - 2\alpha}.$$

Transfer it to the edge version, we get
$$[s_t]_{C^{0, \alpha}_{\text{ie}}} \leq C y(q)^{1 - \alpha}. $$
We may assume $\alpha < \dfrac{1}{2}$, so together with the fact that $|s_t| \leq C y^{\frac{1}{2}}$, we get

$$||s_t||_{y^{\frac{1}{2}} C^{0, \alpha}_{\text{ie}}} \leq C.  $$

Even near the type III boundary or any corner, we have the same estimate. There $y$ should be written as $r\psi$. (We are somehow abuse of notations here. The letter $r$ here means the coordinate near the type III boundary as defined in appendix \ref{Appendix: Compactification of $X$}, not the radius of a ball.) And here is what we get on the entire $X$.

\begin{proposition}
    There is a constant $C$ which depends only on $\Psi_0$ such that for all $t$, uniformly we have
    $$||s_t||_{\rho^{-1}\psi^{\frac{1}{2}}r^{\frac{1}{2}} C^{0, \alpha}_{\text{ie}}} = ||s_t||_{\mathcal{X}^{0, \alpha}_{1/2, 1/2, -1}} \leq C. $$
Moreover, if $t \geq t_0 > 0$ and allow $C$ to depend on $t_0$ and a positive integer $l$, then
$$||s_t||_{\rho^{-l}\psi^{\frac{1}{2}}r^{\frac{1}{2}} C^{0, \alpha}_{\text{ie}}} = ||s_t||_{\mathcal{X}^{0, \alpha}_{1/2, 1/2, -l}} \leq C. $$
\end{proposition}

\paragraph{The $C^{1, \alpha}$ estimates away from boundaries}(Hilderbrandt's estimate)~\\

In the following arguments, sometimes we need to shrink $\alpha$. So we allow $\alpha$ to change from line to line, but always independent with $s_t$. We first work in a region away from the $y = 0 $ boundary (say, $y \geq 1$).\\

Recall that
$$V(\Psi_0, s_t) + ts_t = 0. $$

So $$V(\Psi_0) + L_{\Psi_0}s_t + Q(s_t) + ts_t = 0. $$

Away from boundaries, this can be written as

$$\Delta s_t = A + B(\nabla(s_t)) + C(\nabla s_t \otimes \nabla s_t), $$
where $A, B, C$ are all bounded.

\begin{lemma}\label{Lemma of interior C(1, alpha) estimate}
    Consider the region away from boundaries. Then for some $\alpha \in (0, 1)$,
    $$||\nabla s_t||_{C^{0, \alpha}} \leq C, $$
    where $C$ depends only on $\Psi_0$.
\end{lemma}

\begin{proof}
    Let $B_r$ be a ball of radius $r$ whose center is away from all boundaries. We assume $r \leq 1$.\\

Then from the proof of lemma \ref{lemma of C(0, alpha) estimate 1}, we know

$$\int_{B_r}|\nabla s_t|^2 \leq Cr^{1 + 2\alpha}. $$

Let $s_t = g + h $, where $\Delta h = 0$ in $B_r$ and $h = s_t$ on $\partial B_r$, $\Delta g = \Delta s_t$ in $B_r$ and $g = 0$ on $\partial B_r$.\\

Use $(\widebar{s_t})_{r}$ to represent the average of $s_t$ over the ball $B_r$. That is to say,
$$(\widebar{s_t})_{r} := \dfrac{1}{\text{Vol}(B_r)} \int_{B_r} s_t. $$
Similar definitions for $\bar{g}_r$ and $\bar{h}_r$ etc..\\

Let $f(r) : = \displaystyle \int_{B_r} |\nabla s_t - (\widebar{\nabla s_t})_r|^2$. Using theorem \ref{theorem D2} in the appendix \ref{Appendix: Morrey-Camponato spaces and inequalities}, it suffices to prove that $f(r) \leq Cr^{3 + 2\alpha'}$ for some possibly smaller $0 < \alpha' \leq \alpha$.\\

Note that we already have

$$f(r) \leq Cr^{1 + 2\alpha}. $$

By lemma \ref{lemma D3} in the appendix \ref{Appendix: Morrey-Camponato spaces and inequalities}, for any $0 < r_1 < r$,

$$\dfrac{1}{r_1^3} \int_{B_{r_1}}|\nabla h - \widebar{\nabla h}_{r_1}|^2 \leq C (\dfrac{r_1}{r})^2 \cdot \dfrac{1}{r^3} \int_{B_{r}}|\nabla h - \widebar{\nabla h}_r|^2. $$

On the other hand,
$$\int_{B_r} |\nabla g|^2 = \int_{B_r}<\Delta g, g> ~ \leq ~ C(\int_{B_r} |g| (1 + |\nabla s_t|^2).  $$

Since $||s_t||_{C^{0, \alpha}} \leq C$, we have on the entire $B_r$,
$$|s_t - (\widebar{s_t})_{r}| \leq Cr^{\alpha}. $$

By maximal principle on $B_r$, we have
$$|h - (\widebar{s_t})_{r}| \leq Cr^{\alpha}.$$

So $$|g| = |s_t - h| \leq Cr^{\alpha}. $$

Thus

$$\int_{B_r} |\nabla g|^2 \leq C r^{\alpha} \int_{B_r} |\nabla s_t|^2 \leq Cr^{1 + 3\alpha}. $$

We know for $r \leq R$,

$$f(r) =  \displaystyle \int_{B_{r}} |\nabla s_t - (\widebar{\nabla s_t})_{r}|^2 \leq \int_{B_{r}}|\nabla h - \widebar{\nabla h}_{r}|^2 + \int_{B_{r}}  |\nabla g|^2  $$
$$\leq C (\dfrac{r}{R})^5 f(R) + Cr^{1 + 3\alpha}. $$

If we fix $R = 1$, then this implies
$$f(r) \leq C r^{1 + 3\alpha}. $$

So the Campanato norm $\mathcal{L}^{2, 1+3\alpha}$ of $\nabla s_t$ over $B_1$ is bounded above by $C$.\\

By theorem \ref{theorem D2} in the appendix \ref{Appendix: Morrey-Camponato spaces and inequalities}, this implies that
$$\int_{B_r}|\nabla s_t|^2 \leq C \max\{ r^{1 + 3\alpha}, r^{3}\}. $$

We may use the smaller power between $r^{1 + 3\alpha}$ and $r^3$ as the new starting point and re-run the argument. After finitely many times of iterations, we get for some possibly smaller $\alpha'$,
$$f(r) \leq C r^{3 + 2\alpha'}. $$
Then by theorem \ref{theorem D2} in the appendix \ref{Appendix: Morrey-Camponato spaces and inequalities}, we know that
$$||\nabla s_t ||_{C^{0, \alpha'}} \leq C. $$

\end{proof}

\paragraph{The weighted $C^{1, \alpha}_{\text{ie}}$ estimates near boundaries}(an adapted version of Hilderbrandt's estimate)~\\

Near the $y = 0$ boundary (but away from the $r = 0$ boundary), we have

$$\Delta s_t = A + B(\nabla s_t) + C(\nabla s_t \otimes s_t), $$
where only $C$ is bounded, and $A = O(\dfrac{1}{y^2})$, $B = O(\dfrac{1}{y})$. The fact that $A$, $B$ are not bounded is because $|\Psi_0|$ is not bounded there.\\

Thus on a ball $B$ of radius $\dfrac{y(q)}{4}$ centered at $y(q)$, we want to re-run the proof of lemma \ref{Lemma of interior C(1, alpha) estimate}.\\

To start, we have

$$\int_{B_r}|\nabla s_t|^2 \leq Cy(q)^{\frac{1}{2}}r^{1 + 2\alpha}. $$

Let $R = \dfrac{y(q)}{4}$. And the same argument as in the proof of lemma \ref{Lemma of interior C(1, alpha) estimate} leads to (we still have the same definition of $f(r)$)

$$f(r) \leq C (\dfrac{r}{R})^5f(R) + C(y(q))^{\frac{1}{2}}r^{1 + 3\alpha} \leq C y(q)^{\frac{1}{2} - \alpha} r^{1 + 3\alpha}. $$

As long as the exponent of $r$ is less or equal than $3$ we can run the same argument and gain an extra $y(q)^{-\alpha} r^{\alpha}$ factor. After finitely many times of iterations, this will lead to

$$f(r) \leq C y(q)^{- \frac{3}{2}}r^{3 + 2\alpha'}, $$
for some possibly smaller $\alpha'$.\\

This implies the $y(q)^{\frac{1}{2}} C^{0, \alpha'}$ norm of $y(q) \nabla s_t$  is bounded above by $C$ on this ball. Note that $C$ doesn't depend on the ball. So we conclude that, in a region near the $y = 0$ boundary away from $r = 0$ boundary and for possibly smaller $\alpha$, we have
$$||s_t||_{y^{\frac{1}{2}} C^{1, \alpha}_{ie}} \leq C. $$

Even when approaching the $r = 0$ boundary, the argument doesn't need to change and we get locally

$$||s_t||_{r^{\frac{1}{2}}\psi^{\frac{1}{2}} C^{1, \alpha}_{ie}} \leq C. $$

On the other hand, near the $\rho \rightarrow +\infty$ boundary, we have

$$\Delta \rho^l s_t = A + B(\nabla^l \rho s_t) + C(\nabla^l \rho s_t \otimes \rho^l s_t), $$
where $A, B, C$ are bounded, $l = 1$ if $t \in [0, 1]$ and $l$ can be any integer if $t \geq t_0 > 0$ with the bounds of $A, B, C$ depends on $t_0$ and $l$.\\

So examine the proof of lemma \ref{Lemma of interior C(1, alpha) estimate}, we get:
    $$||s_t||_{\mathcal{X}^{1, \alpha}_{1/2, 1/2, -1}} \leq C, $$
where $C$ depends only on $\Psi_0$. 
When $t \geq t_0 > 0$ and when $C$ is allowed to depend on $t_0$, the positive integer $l$, then

$$||s_t||_{\mathcal{X}^{1, \alpha}_{1/2, 1/2, -l}} \leq C.$$

All the higher weighted $C^{k, \alpha}_{\text{ie}}$ norms with $k \geq 2$ follow by the standard elliptic Holder type interior estimates. So to conclude we have

\begin{proposition}
    There exists an $\alpha \in (0, 1)$  such that
    $$||s_t||_{\mathcal{X}^{k, \alpha}_{1/2, 1/2, -1}} \leq C. $$
Here the constant $C$ depends only on $\Psi_0$ and the positive integer $k$. 
When $t \geq t_0 > 0$ and we assume $C$ depends only on $t_0$, $\Psi_0$, positive integers $k$ and $l$, we have stronger estimate:

$$||s_t||_{\mathcal{X}^{k, \alpha}_{1/2, 1/2, -l}} \leq C.$$
\end{proposition}

Note that this proposition is the third and fifth bullets of the facts listed in subsection \ref{Subsection: The continuity argument}.

\subsection{More regularities}\label{Subsection: More regularities}

This subsection proves the second and the fourth bullet of the facts listed in the end of subsection \ref{Subsection: The continuity argument}.\\

Suppose $s$ is a solution for some $t$. We hope to study the operator $L_{s}: = L_{\Psi_s}$. Unfortunately this operator seems to be hard to study. However, suppose $t > 0$. Consider alternatively the operator $L_{s, t} : = L_{s} + t$. Then it behaves much better and can be analysis-ed.\\

In fact, we have the following lemma.

\begin{lemma}
    
Assume $t > 0$ and $s$ in $\mathcal{X}^k_{\frac{1}{2}, \frac{1}{2}, -l}$ (for all positive integers $k, l$) is a solution of
$$V(\Psi_0, s) + ts = 0. $$
Suppose $F \in \mathcal{X}^{k}_{-\frac{3}{2} + \epsilon, -\frac{3}{2} + \epsilon, - l}$ for some small $\epsilon > 0$ and all positive integers $k$ and some large enough $l > 0$. Then there exists a $u \in \mathcal{X}^k_{\frac{1}{2}, \frac{1}{2}, -l}$ for all $k$ such that 
$$L_{s, t}(u) = F. $$
Moreover, the $\mathcal{X}^{k+2}_{\frac{1}{2}, \frac{1}{2}, -l}$ norm of $u$ is bounded above by (a constant times) the $\mathcal{X}^{k}_{-\frac{3}{2}+\epsilon, -\frac{3}{2}+\epsilon, -l}$ norm of $F$.
\end{lemma}

\begin{proof}
    We use two cut-off functions $\chi_1, \chi_2$ to divide $F$ into two parts: $F_i = \chi_i F$, $i = 1, 2$. Here $\chi_i$ are all smooth functions with $$\chi_1 + \chi_2 = 1.$$
    We assume $\chi_1$ is supported in a region such that $y < \min\{2, 2R^{-l}\}$. And $\chi_2$ is supported in a region such that either $y > \min\{1, R^{-l}\}$. Clearly we only need to solve $L_{s, t}(u_i) = F_i$ on each region.\\

We first consider $F_1$. We apply Mazzeo's theory of elliptic edge operators (and terminologies) here, which can be found in \cite{Mazzeo1991EllipticI}. We quote theorem 5.8 in \cite{He2020TheII} (where the indicial weights are actually computed in \cite{Mazzeo2017TheKnots}).  In fact, when $y \rightarrow 0$, the operator $L_{s, t}$ has the same normal operator as if $t = 0$ and $\Psi$ is the model solution in their construction. Here is the only fact that we need: The weight $r^{\frac{1}{2}}\psi^{\frac{1}{2}}$ lies in the Fredholm range both near the $r \rightarrow 0$ boundary and the $\psi \rightarrow 0$ boundary (and is compatible with the corners). (Alternatively, readers may compute the Fredholm weight of $L_{s, t}$ directly. It is actually easier than Mazzeo and Witten's computations in \cite{Mazzeo2017TheKnots}, which takes more situations into account.) In particular, near those boundaries,

$$L_{s, t}(u) = F_1$$

has a solution in the support of $\chi_1$ with $u \in r^{\frac{1}{2}}\psi^{\frac{1}{2}}C^{k}_{\text{ie}}$ for all $k$. Here we allow that $u$ is nonzero in a slightly larger region and doesn't satisfy the equation outside of the support of $\chi_1$. But we may add $L_{s, t}(u) - F_1$ (which is supported in a relatively compact region) to $F_2$ there and throw it into the next step.\\

Here is the next step: We solve the equation for $F_2$.  Consider the following functional on $u$, where $u \in \mathcal{X}^k_{\frac{1}{2}, \frac{1}{2}, -l}$ for all $k, l$:

$$A(u): = \int_{X} (\sum\limits_{i = 1, 2, y}|\nabla_i u|^2 + \sum\limits_{i = 1, 2, 3} |[\Phi_i, u]|^2 + <F_2, u>). $$

We may first consider the Banach space $\IH$ defined by completion of smooth compact supported functions using the following norm:

$$||u||_{\IH}^2 = \int_{X} (\sum\limits_{i = 1, 2, y}|\nabla_i u|^2 + \sum\limits_{i = 1, 2, 3} |[\Phi_i, u]|^2). $$

Clearly, $A(u)$ is bounded in $\IH$ and has a Dirichlet minimizor in $\IH$. (Recall that $|F_2| \leq C R^{-l}$ for sufficiently large $l$.) This minimizor is unique because of the convexity of $A(u)$. By a standard elliptic regularity argument, this minimizor, as a week solution of $L_{s, t}(u) = F_2$, is smooth in the interior of $X$. Moreover, there is a Hardy type of inequality for $u \in \IH$:
$$\int_{X} y^{-2} |u|^2 \leq C \int_{X} \sum\limits_{i = 1, 2, y}|\partial_i u|^2 \leq \int_{X} \sum\limits_{i = 1, 2, y} |\nabla_i u|^2 \leq ||u||_{\IH}^2 < +\infty.$$

In particular, the integral of both $|u|$ and $|\nabla u|$ over the region $R < \lambda$ has at most a polynomial growth as $\lambda \rightarrow +\infty$. We have

$$|<L_{t, s}u, u>| = |<F_2, u>| \leq |F_2| |u|.  $$

But
$$<L_{t, s}u, u> = <- \Delta_{\Psi}u + tu, u> = - \dfrac{1}{2}\Delta (|u|^2) + \sum\limits_{i = 1, 2, y} |\nabla_{A_i}u|^2 + \sum\limits_{i = 1, 2, 3} |[\Phi_i, u]|^2 + t|u|^2 $$
$$\geq - \dfrac{1}{2} \Delta(|u|^2) + |\nabla u|^2 + t|u|^2 = - <\Delta u, u> ~~~\geq ~~~ (- \Delta |u| + t|u|)|u|. $$

So we get

$$- \Delta |u| + t|u| \leq |F_2|. $$

Let $\chi_{\lambda}$ be a cut-off function that equals $1$ when $R \leq \lambda$ and equals $0$ when $R \geq 2\lambda$. Moreover, we assume $|\nabla \chi_{\lambda}| = O(R^{-1}),~ |\nabla^2 \chi_{\lambda}| = O(R^{-2}).$ And let $u_{\lambda} := u \chi_{\lambda}$. Then recall the Green's function of $-\Delta + t$ is $G_{t, q}$ (defined in subsection \ref{Subsection: A priori estimates}) and the fact that $u_{\lambda}$ has compact support, we have (centering at any point $q \in X$)
$$|u_{\lambda}(q)| \leq \int_{X} G_{t, q} \Delta |u_{\lambda}| \leq C\int_{\text{support}(\nabla \chi_{\lambda})}G_{t, q} (R^{-2}|u| + R^{-1}|\nabla u|) + \int_{X}G_{t, q} |F_2|). $$
Note that when $\lambda \rightarrow +\infty$, $G_{t, q}$ has exponential decay while the integral of $|u|$ and $|\nabla u|$ on support$(\nabla \chi_{\lambda})$ have at most polynomial growth. So

$$\lim\limits_{\lambda \rightarrow +\infty}\int_{\text{support}(\nabla \chi_{\lambda})}G_{t, q} (R^{-2}|u| + R^{-1}|\nabla u|) = 0.$$

And letting $\lambda \rightarrow +\infty$, 
$$|u(q)| \leq \int_{X} G_{t, q}|F_2|. $$

We may divide $X$ into two parts: Let $B$ be a ball of radius $\dfrac{R(q)}{2}$, where $R(q)$ is the $R$ value of $q$. Then

$$\int_{X} G_{t, q}|F_2| = \int_{X\cap B} G_{t, q}|F_2| + \int_{X\backslash B} G_{t, q}|F_2|. $$

When $R(q) \rightarrow +\infty$, the second integral above decays exponentially in $R(q)$. The first integral is bounded by (recall that $|F_2| \leq CR^{-l}$)

$$CR^{-l}\int_{B\cap X}G_{t, q} \leq CR^{-l}\int_{ X}G_{t, q} \leq CR^{-l}.  $$

So we get 

$$|u| \leq CR^{-l}. $$

Once we have this point-wise bound, then the bound on Holder norms $\rho^{-l}C^{k, \alpha}$ of $u$ follows by a standard elliptic interior argument and are omitted. (It is actually much easier than the analysis in subsection \ref{Subsection: A priori estimates} because we don't have the quadratic term here.)\\

Finally, we add the two different $u$ that we get for $F_1$ and $F_2$ together and finish the proof.

\end{proof}

Suppose for some $t > 0$, we have a solution $s$ with

$$V(\Psi_0, s) + ts = 0.$$
Consider a small variation $\epsilon \alpha$ on top of $s$ and a small variation $-\epsilon$ on top of $t$:

$$V(\Psi_0, s + \epsilon \alpha) + (t - \epsilon) (s + \epsilon \alpha) = \epsilon (L_{s, t}(\alpha) - s) + \epsilon^2 (Q(\alpha) - \alpha).$$

Note that $Q(\alpha) - \alpha$ is a bounded map from $\mathcal{X}_{\frac{1}{2}, \frac{1}{2}, -l}^{k+2, \alpha}$ to $\mathcal{X}^{k, \alpha}_{-\frac{3}{2} + \epsilon, -\frac{3}{2} + \epsilon, - l}$ . So when $\epsilon$ is small enough, by implicit function theorem, the equation
$$L_{s, t}(\alpha) = s + \epsilon(\alpha - Q(\alpha)) $$
has a solution in $\mathcal{X}^k_{\mu, v, -l}$. This proves the second bullet of the facts in subsection \ref{Subsection: The continuity argument}.\\

Suppose when $t > 0$, $s_t$ is a solution in $\mathcal{X}^{k, \alpha}_{\frac{1}{2} - \epsilon, \frac{1}{2} - \epsilon, -l}$. We have
$$V(\Psi_0) + L_{0, t}(s_t) + Q(s_t) = 0,$$
where $Q(s_t)$ actually maps $\mathcal{X}_{\frac{1}{2} - \epsilon, \frac{1}{2}-\epsilon, -l}^{k+2, \alpha}$ into $\mathcal{X}^{k, \alpha}_{-1 - 2\epsilon, -1 - 2\epsilon, - l}$ for sufficiently small $\epsilon > 0$. In particular, $Q(s_t)$ is in $\mathcal{X}^{k, \alpha}_{-\frac{3}{2}, -\frac{3}{2}, - l}$. Then because $\mu \in [\dfrac{1}{2} - \epsilon, \dfrac{1}{2}]$ and $v \in [\dfrac{1}{2} - \epsilon, \dfrac{1}{2}]$ are all Fredholm weights for $L_{0, t}$ near those boundaries, and $V(\Psi_0)$ vanishes up to infinite order at all boundaries. So inductively $s_t$ lies in $\mathcal{X}^{k, \alpha}_{\frac{1}{2}, \frac{1}{2}, -l}$ for all $k, l$. This proves the fourth bullet of the facts in subsection \ref{Subsection: The continuity argument}.

\appendix

\section{The linear algebra}\label{Appendix: The linear algebra}

This appendix summarizes the linear algebras that we use.\\

Suppose $E$ is the trivial $\IC^2$ bundle whose $SL(2, \IC)$ structure is fixed. Since we are working on the Euclidean space, we take the advantage that nearly everything can be represented by matrices. In the following list, all matrices have complex variable items. Note that we use $*$ to represent the usual complex ad-joint (conjugate of the transpose) of a matrix.

\begin{itemize}
    \item An \textbf{$SU(2)$ Hermitian metric} on $E$ is represented by an $2 \times 2$ matrix $H$ such that $H^* = H$, $\det H = 1$, and positive definite. Each such metric gives $E$ an $SU(2)$ structure. (The condition $\det H = 1$ keeps the $SL(2, \IC)$ structure unchanged.) The inner product of two sections $s_1, s_2 \in \Gamma(E)$ (represented by $2$-d vectors with variable coefficients) is
    $$<s_1, s_2>_H = \dfrac{1}{2} \tr (H^{-1}s_1^*Hs_2 + H^{-1}s_2^*Hs_1). $$
    Unless otherwise specified, we typically just use the inner product defined by $H = I$, that is
    $$<s_1, s_2> = \dfrac{1}{2} \tr (s_1^*s_2 + s_2^*s_1) $$
\item Infinitesimal $SL(2, \IC)$ gauge transformations are represented by sections of $sl(2, \IC)$. Infinitesimal $SU(2)$ gauge transformations are represented by sections of $su(2)$. And $sl(2, \IC) = su(2) \oplus isu(2)$, where $isu(2)$ are the Hermitian $sl(2, \IC)$ elements.
\item One useful formula: Suppose $s(t)$ is a differentiable $1-$parameter family of matrices. Then
$$\dfrac{d}{dt} e^{s} = e^{s} ~ \gamma(-s) (\dfrac{ds}{dt}) = \gamma(s) (\dfrac{ds}{dt}) ~ e^{s},$$
where $$\gamma(s) (M) = (\dfrac{e^{ad_s} - 1}{ad_s}) (M) = \sum\limits_{k=0}^{+\infty} \dfrac{(ad_s)^{k}}{(k + 1)!} (M),$$ $$ad_s(M) = [s, M] = sM - Ms.$$
\begin{proof} Let $L_s(M) = sM, R_s(M) = Ms.$ Then
$L_s, R_s, ad_s$ commutes. And $L_s = R_s + ad_s$. Let $C_n^m = \dfrac{n!}{m!(n-m)!}$. Then
    $$\dfrac{d}{dt}(e^s) = 
\dfrac{d}{dt}(\sum\limits_{n = 0}^{+\infty} \dfrac{s^n}{n!}) = \sum\limits_{n = 0}^{+\infty} \dfrac{1}{n!}(\sum\limits_{k = 0}^{n-1} s^{n-1-k} (\dfrac{ds}{dt}) s^{k}) = \sum\limits_{n = 0}^{+\infty} \dfrac{1}{n!}(\sum\limits_{k = 0}^{n-1} R_s^k L_s^{n-1-k} (\dfrac{ds}{dt}))  $$
$$~~~~~~~~~~~~~~~~~~~~= \sum\limits_{n = 0}^{+\infty} \sum\limits_{k = 0}^{n-1} \dfrac{1}{n!} \sum\limits_{l=0}^{n - 1 - k} C_{n-1-k}^l(R_s^{l+k} (ad_s)^{n - 1 - k - l}) (\dfrac{ds}{dt}) $$
$$(\text{let} ~~ v = l+k) ~~ = \sum\limits_{n = 0}^{+\infty} \dfrac{1}{n!} \sum\limits_{v = 0}^{n-1} (\sum\limits_{l = 0}^v C_{n-1-v + l}^l) R_s^v (ad_s)^{n - 1 - v} (\dfrac{ds}{dt})   $$
$$ = \sum\limits_{n = 0}^{+\infty} \dfrac{1}{n!} \sum\limits_{v = 0}^{n-1} C_n^{v} R_s^v (ad_s)^{n - 1 - v} (\dfrac{ds}{dt})   $$
$$(\text{let} ~~ j = n - 1 - v) ~~ = (\sum\limits_{v = 0}^{+\infty} \dfrac{1}{v!} R_s^v) (\sum\limits_{j = 0}^{+\infty} \dfrac{1}{(j+1)!} (ad_s)^j) (\dfrac{ds}{dt}) = \gamma(s) (\dfrac{ds}{dt}) ~~ e^{s}.$$
The other identity can be derived in the same way using $R_s = L_s + ad_{-s}$.
\end{proof}

\item Another useful formula: Suppose $s, M$ are two $2 \times 2$ matrices. Then
$$e^{-s}Me^s = e^{ad_{-s}}(M) = M + (\dfrac{e^{ad_{-s}} - 1}{ad_{-s}}) (ad_{-s}M) = M + \gamma(-s) ([M, s]). $$

\begin{proof}
    This is straightforward: 
$$e^{-s}M e^s = \sum\limits_{k,l=1}^{+\infty} \dfrac{1}{k!l!}(-1)^k (L_s)^k(R_s)^l M = \sum\limits_{v = 0}^{+\infty} \sum\limits_{k = 0}^v (-1)^k(L_s)^k(R_s)^{v-k} M $$
$$ = \sum\limits_{v = 0}^{+\infty} \dfrac{1}{v!} C_v^{k}(-1)^k(L_s)^k(R_s)^{v-k}M = \sum\limits_{v = 0}^{+\infty} \dfrac{1}{v!} (R_s - L_s)^vM = e^{ad_{-s}}M.$$
\end{proof}

\item Suppose $s$ is Hermitian. Then the operator $\gamma(s) = (\dfrac{e^{ad_s}- 1}{ad_s})$ has a square root:
$$v(s) = \sqrt{\gamma(s)} = \sqrt{\dfrac{e^{ad_s} - 1}{ad_s}}. $$
In fact, consider the function $f_p(x) = (\dfrac{e^x - 1}{x})^p = \exp(p \cdot \ln (\dfrac{e^{x} - 1}{x}))$, where $p$ is any real number. Clearly $f_p(x)$ is a real analytical function over $x \in \IR$. It has a convergent Taylor series expansion at any point. On the other hand, when $s$ is Hermitian, $ad_{s}$ is also a self-adjoint operator. So it is diagonalizable  and has real eigen-values.  In particular, when $p = \dfrac{1}{2}$, we use this expansion to define $$v(s) = f_{1/2}(ad_s).$$ 
\end{itemize}

\section{Compactification of $X$}\label{Appendix: Compactification of $X$}

This appendix gives a preferred way to compactify $X$ as a manifold with boundaries and corners $\hat{X}$. In this paper, we have always assumed that $X$ is compactified in this way whenever we work near one of the boundaries/corners and a compactification is needed.\\

There are three types of boundaries of the compactification of $X$. We call them ``\textbf{type I, II, III boundaries}" respectively. Type I and type II boundaries intersect at a \textbf{type A corner}. Type II and type III boundaries interest at a \textbf{type B corner}. Type I and type III boundaries do not intersect.\\

Here are the definitions and local coordinates.

\paragraph{Type I boundary ($\rho \rightarrow  + \infty$, or equivalently $R \rightarrow +\infty$)}~\\

Recall that $R^2 = |z|^2 + y^2$. Here is the definition of $\rho$: It is a smooth function from $X$ to $[1, +\infty)$ that equals $1$ when $R$ is small and equals $R$ when $R$ is large.\\

So $\rho = +\infty$ (or equivalently $\dfrac{1}{\rho} = 0$) defines a boundary for the compatification $X$. This is the type I boundary.

\paragraph{Type III boundaries ($r = 0$)}~\\

Here is the definition of $r$: It is a smooth function from $X$ to $(0, 1]$.\\

Suppose $z_0$ is any root of $P(z)$ (which corresponds to a ``knotted point" at $z = z_0, y = 0$ of the generalized Nahm pole boundary condition). When both $y$ and $|z - z_0|$ are small, we require $r^2 = |z - z_0|^2 + y^2$. When either $y$ is large or $z$ is away from all roots of $P(z)$, we require that $r = 1$.\\

Then $r = 0$ defines the type III boundaries of the compatification of $X$. Note that we have blowed up at $(z_0, 0)$ for each root $z_0$ of $P(z)$.

\paragraph{Type II boundary ($\psi = 0$, or when $\rho = r = 1$ equivalently $y = 0$ there)}~\\

Here is the definition of $\psi$: It is a smooth function from $X$ to $(0, 1]$.\\

Away from the type I and type III boundaries (say, $r = \rho = 1$), when $y$ is small, we require $\psi = y$. When $y$ is large, we require $\psi = 1$.\\

Near a type I boundary, if $\dfrac{y}{R}$ is small, then we require $\psi = \dfrac{y}{R}$.  When $\dfrac{y}{R}$ is close to $1$, we require that $\psi = 1$.\\

Near a type III boundary, we require $\psi = \dfrac{y}{r}$.\\

Note that effectively, away from other boundaries, $\psi = 0$ and $y = 0$ define the same boundary. But we  use $\psi$ instead of $y$ because it is also compactible with other boundaries.

\paragraph{Type A and type B corners}~\\

Type A corners are given by $\psi = 0, \rho = +\infty$ (or equivalently $\psi = \dfrac{1}{\rho} = 0$). And type B corners are given by $\psi = r = 0$.\\

\paragraph{A remark on the coordinates:}

When $\rho$ is large, since $\rho = R$, we may freely choose to use either $R$ or $\rho$ there for the same meaning. But usually we use $\rho$ if we want to emphasize that it equals $1$ (not arbitrarily small) when $R \rightarrow 0$.\\

When $r$ is not too small and $\rho$ is not too large and when $y$ is small, $y$ and $\psi$ can bound each other. So  they are also interchangeable there in analysis.\\







\section{Weighted Holder spaces of (iterated) edge type}\label{Appendix: Weighted Holder spaces of (iterated) edge type}

This appendix defines the Banach spaces $\mathcal{X}^{k, \alpha}_{\mu, v, l}$, where $\alpha \in (0, 1)$, $k$ is a non-negative integer, $\mu, v, l$ are real numbers. These spaces are standard in the aspect of Mazzeo's micro-local analysis theory (see \cite{Mazzeo1991EllipticI}). They've also occurred in \cite{Mazzeo2017TheKnots}, \cite{He2020TheII} and \cite{Dimakis2022TheField}. (For the sake of convenience, the descriptions here may be modified compared to other literature in a non-essential way.)

\subsection{The Holder spaces of (iterated) edge type}

Suppose $B$ is a ball in $X$ far away from any boundary/corner (the distance to any boundary/corner is at least $1$). Then we define the Holder spaces $C^{k, \alpha}(B)$ over $B$, where $k$ is a non-negative integer and $\alpha \in (0, 1)$.  This space is given by the norm:

$$||u||_{C^{k, \alpha} (B)} : =  \sum\limits_{j = 0}^{k}||\nabla^j u||_{L^{\infty}(B)} + [\nabla^k u ]_{C^{0, \alpha}(B)}, $$
where
$$[u]_{C^{0, \alpha}(B)} : = \sup\limits_{p, q \in {B}, p \neq q} \dfrac{|u(p) - u(q)|}{|p - q|^{\alpha}}. $$

In a region far away from type II ($\psi = 0$) and type III ($r = 0$) boundaries, we take the supreme of all balls of radius $1$ of the above norm.\\

The operator that we want to study is $L_{s, t}$ which is introduced in subsection \ref{Subsection: More regularities}, where $0 < t \leq 1$.  This operator is of the ``degenerate elliptic of (iterated) edge type" near a type II or a type III boundary as studied in \cite{Mazzeo2017TheKnots}, \cite{He2020TheII} and \cite{Dimakis2022TheField}. It is standard to modify the Holder spaces near those boundaries. The modified version will be denoted as $[\cdot]_{C^{0, \alpha}_{\text{ie}}}$.\\

We take the boundary $r = 0$ as an example. One way to think about the modification is to re-define the distance between two points $p$ and $q$ near the boundary to make it re-scaling invariant under a re-scaling $r \rightarrow \lambda r$. This is done by modifying the metric near the boundary. In a direction that is perpendicular with the $\partial_{\psi}$ direction, there is nothing need to be changed. But in the $\partial_{r}$ direction, the metric should be $\dfrac{1}{r^2} dr^2$ instead of $dr^2$. Thus the distance between a point at $r_1$ and a point at $r_2$ (with all other perpendicular coordinates the same) is given by

$$|\int_{r_1}^{r_2} \dfrac{1}{r} dr | = |\ln(r_1) - \ln(r_2) |,$$

which is clearly re-scaling invariant.\\

Another equivalent way to do the modification is to define it on each ball $B$ of radius $\dfrac{r_0}{2}$ whose center has an $r$-value $r_0$. On this ball, the $[\cdot] _{C_{\text{ie}}^{0, \alpha}}$ of $u$ should be given by

$$\sup\limits_{p, q \in B, p \neq q} \dfrac{r_0^{\alpha}|u(p) - u(q)| }{|p - q|^{\alpha}}. $$

Note that the weight $r_0^{\alpha}$ here works equivalently as if the metric is re-scaled in this ball. (They bound each other in a way that doesn't depend on $r_0$ near the boundary.)\\

When it comes to the $\psi = 0$ boundary but away from the corner (or equivalently, $y = 0$ when $r$ is not too small and $|z|$ is not too large), things are slightly different. The operator $L_{\Psi} + t$ (strictly speaking, its $y^2(L_{\Psi} + t)$) has leading order terms which are made from combination of compositions of  $y\partial_y, y\partial_1, y\partial_2$. So the re-scaling should be made in both $y$ and $z$ direction. And the metric should be addapted to be $\dfrac{1}{y^2}(dy^2 + dx_1^2 + dx_2^2)$ near the boundary.\\

Similarly, an equivalent way is to define it on each ball $B$ of radius $\dfrac{y_0}{2}$ centered at a point whose $y$ value is $y_0$. On this ball $B$, the $[\cdot]_{C^{0, \alpha}_{\text{ie}}}$ of $u$ is given by

$$\sup\limits_{p, q \in B, p \neq q} \dfrac{y_0^{\alpha}|u(p) - u(q)| }{|p - q|^{\alpha}}. $$

Near the corner, the metric is adjusted so it is dual re-scaling invariant in two directions that corresponds to the two boundaries. \\

Note that we do not need to modify the Holder norm when $\rho \rightarrow +\infty$ (type I boundary). Because the operator $L_{\Psi} + t$ is not of the ``degenerate elliptic of edge type" near this boundary.\\

For the higher Holder spaces, in the definition of $C^{k, \alpha}_{\text{ie}}$, near $r = 0$ boundary, we need to replace $\partial_r u$ in $\nabla u$ by $r \partial_r u$. And near $\psi = 0$ boundary (or $y = 0$ boundary but away from $r = 0$ boundary), we need to replace $\nabla u$ by $y \nabla u$. This corresponds to the edge structure of the operator $L_{\Psi} + t$ that we study. Nothing needs to be adjusted when $\rho \rightarrow +\infty$.

\subsection{The weighted Holder spaces}

In order to be suitable for the operator $L_{s, t}$ (occurred in subsection \ref{Subsection: More regularities}) to map between, we need to add wights near boundaries/corners of the aforementioned Holder spaces $C^{k, \alpha}_{\text{ie}}$. Here is the definition:

$$\mathcal{X}^{k, \alpha}_{\mu, v, \delta} : = \psi^{\mu}r^{v}\rho^{\delta} C^{k, \alpha}_{\text{ie}} = \{\psi^{\mu}r^{v}\rho^{\delta}u ~ | u \in C^{k, \alpha}_{\text{ie}}\}.$$

Here are several remarks:

\begin{itemize}
    \item If $u$ is an element in $\mathcal{X}^{k, \alpha}_{\mu, v, \delta}$ for some $\alpha \in (0, 1)$ and all positive integers $k$, then $u$ is also in $\mathcal{X}^{k, \alpha'}_{\mu, v, \delta}$ for any other $\alpha \in (0, 1)$. If this is the case, then the concrete $\alpha$ doesn't matter and we may simply use $\mathbf{X}^k_{\mu, v, \delta}$ to represent it.
    \item Although all these norms and spaces are defined for functions, frequently we use them on sections of trivial bundles. The difference is only tautological so we don't emphasize it.
    \item If $v = \mu$, then 
    $\psi^{\mu} r^{\mu} C^{k, \alpha}_{\text{ie}}$ is actually the same space as $y^{\mu} C^{k, \alpha}_{0}$, where we treat $y = 0$ as a single boundary (without the blow-ups at each $r = 0$ point like what we did in appendix \ref{Appendix: Compactification of $X$}), and the subscript ``0" means this is the ordinary edge type Holder norm as $y \rightarrow 0$, not iterated edge type.
\end{itemize}

\section{Morrey-Camponato spaces and inequalities} \label{Appendix: Morrey-Camponato spaces and inequalities}

The Morrey-Camponato spaces and their embedding inequalities are standard in analysis. We only state what we need. We always assume $B$ is a ball of radius $R$ whose closure is in the interior of $X$.\\

Note that although the spaces and inequalities are stated for functions, they work the same tautologically for sections of trivial vector bundles. So we don't emphasize the difference.

\begin{definition}\label{Definition D1}
   Suppose $u$ is a function in $X$ and $\lambda$ is a real number. Let $B_r(x)$ be the ball of radius $r \leq R$ centered at $x \in B$. Then the Morrey norm of $u$ is defined to be
$$||u||_{L^{2, \lambda}(B)} = (\sup\limits_{B_r(x)} r^{-\lambda} (\int_{B_r(x) \cap B}|u|^2))^{\frac{1}{2}}. $$
The Camponato semi-norm is
$$[u]_{\mathcal{L}^{2, \lambda}(B)} = (\sup\limits_{B_r(x)} r^{-\lambda} (\int_{B_r(x) \cap B}|u - \bar{u}_{r, x}|^2))^{\frac{1}{2}},$$
where $\bar{u}_{r, x}$ is the average of $u$ in the ball $B_r(x)$, that is to say
$$\bar{u}_{r,x} := \dfrac{1}{\text{Vol}(B_r(x) \cap B)} \int_{B_r(x) \cap B} u.  $$

The Camponato norm is

$$||u||_{\mathcal{L}^{2, \lambda}(B)} = (\dfrac{1}{\text{Vol}(B)}\int_{B } |u|^2)^{\frac{1}{2}} + [u]_{\mathcal{L}^{2, \lambda}(B)}. $$

\end{definition}

We have some embedding theorems between different normed spaces. These are all standard in modern analysis so we omit the proofs. Here they are:

\begin{theorem}\label{theorem D2}  Suppose $\lambda > 0$ is a real number. There is a constant $C > 0$ which depends on $\alpha$. Suppose $u$ is a function in $X$. Then
 \begin{itemize}
      \item If $0 < \lambda < 1$, then $[u]_{C^{0, \lambda}(B)} \leq C ||\nabla u||_{L^{2, 1+2\lambda}(B)}$, where $C^{0, \lambda}$ is the Holder (semi-)norm.
      \item $[u]_{\mathcal{L}^{2, \lambda + 2}(B)} \leq C ||\nabla u||_{L^{2, \lambda}(B)}$.
      \item If $\lambda \leq 3$, then
      $||u||_{L^{2, \lambda}(B)} \leq C ||u||_{\mathcal{L}^{2, \lambda}(B)}$ 
\end{itemize} 
 
\end{theorem}

Note that under a re-scaling, the two sides of all the inequalities scale in the same way. So the constant $C$ doesn't depend on the radius of the ball.\\

There is another inequality which is standard in analysis
\begin{lemma}\label{lemma D3} (Lemma 10.3.1 in \cite{Han2011EllipticLin})
We use the same notations as in the definition \ref{Definition D1}.
    Suppose $u$ is a smooth function on $B$ such that
    $$\Delta u = 0.$$
    Suppose $x$ is the center of the ball $B$. Then there is a constant $C$ such that for any $0 < r_1 < r_2 < R$
 $$\int_{B_{r_1}} |u - \bar{u}_{x, r_1}|^2 \leq C (\dfrac{r_1}{r_2})^5 \int_{B_{r_2}} |u - \bar{u}_{x, r_2}|^2.$$
\end{lemma}

\bibliographystyle{plain}
\bibliography{references}

\end{document}